\documentclass[11pt]{amsart}
\usepackage{amsmath,amssymb,amsthm,amsfonts,tikz,mathtools, multicol,fleqn,hyperref}
\usepackage{comment}
\usetikzlibrary{decorations.pathmorphing,cd}
\usepackage[numbers]{natbib}
\newtheorem{thm}{Theorem}[section]
\newtheorem{cor}[thm]{Corollary}
\newtheorem{lem}[thm]{Lemma}
\newtheorem*{Cpct}{Lemma \ref{compactness_emb_degrees}}
\newtheorem*{Duo}{Proposition \ref{compactness_duo}}
\newtheorem*{FmlaFree}{Corollary \ref{cor_fmla_free}}
\newtheorem{prop}[thm]{Proposition}

\theoremstyle{remark}
\newtheorem{rmk}[thm]{Remark}

\newtheorem{ex}[thm]{Example}
\newtheorem{obs}[thm]{Observation}

\theoremstyle{definition}
\newtheorem{dfn}[thm]{Definition}
\newtheorem{qu}[thm]{Question}

\newcommand{\la}{\left <}
\newcommand{\ra}{\right >}

\newcommand{\M}{\mathcal{M}}
\newcommand{\R}{\mathcal{R}}

\newcommand{\Q}{\mathbb{Q}}
\newcommand{\N}{\mathcal{N}}
\newcommand{\Nn}{\mathbb{N}}
\newcommand{\I}{\mathcal{I}}
\newcommand{\Z}{\mathbb{Z}}

\newcommand{\A}{\mathcal{A}}
\newcommand{\B}{\mathcal{B}}
\newcommand{\C}{\mathbf{C}}
\newcommand{\D}{\mathbf{D}}
\newcommand{\K}{\mathcal{K}}

\newcommand{\uphp}{\upharpoonright}
\newcommand{\ov}{\overline}

\newcommand{\raw}{\rightarrow}

\newcommand{\ii}{\ov{\imath}}
\newcommand{\jj}{\ov{\jmath}}

\newcommand{\emtp}{\textrm{EMtp}}
\newcommand{\tp}{\textrm{tp}}
\newcommand{\len}{<_{\textrm{len}}}
\newcommand{\lx}{<_{\textrm{lex}}}
\newcommand{\fr}{Fra\"{i}ss\'{e}}

\newcommand{\tr}{\textrm{stree}}
\newcommand{\str}{\textrm{strtree}}
\newcommand{\eq}{\textrm{eq}}

\newcommand{\age}{\textrm{age}}
\newcommand{\Aut}{\textrm{Aut}}
\newcommand{\smf}{\smallfrown}
\newcommand{\Th}{\textrm{Th}}
\newcommand{\ran}{\textrm{ran}}

\newcommand{\qftp}{\textrm{qftp}}
\newcommand{\Wo}{\omega^{<\omega}}

\newcommand{\snb}{S_n^\B(\emptyset)}
\newcommand{\sna}{S_n^\A(\emptyset)}

\newcommand{\CC}{\mathbf{C}}
\newcommand{\id}{\textrm{id}}

\newcommand{\rqe}{\texttt{rqe}}

\newcommand{\LA}{L({\A})}
\newcommand{\LB}{L({\B})}
\newcommand{\Diag}{\textrm{Diag}}
\newcommand{\emb}{\textrm{Emb}}
\newcommand{\obj}{\textrm{Obj}}

\newcommand{\Ba}{\mathcal{B}_{ba}}
\newcommand{\EK}{\mathcal{E}\mathcal{K}}

\newcommand{\CEKS}{\mathcal{C}\mathcal{E}[\mathcal{K}^*]}

\title{A New Perspective on Semi-retractions and the Ramsey Property}
\thanks{The first author was supported by National Science Foundation grants DMS-1953955 and CAREER DMS-2144118.}
\thanks{The second author was supported by National Science Foundation grant number DMS-2246995.}
\thanks{This article is based upon work supported by the National Science Foundation under Grant No.~DMS-1928930 while the second author participated in a program hosted by the Mathematical Sciences Research
Institute in Berkeley, California, during Summer 2022.  The second author completed parts of this paper in collaboration with the first author while supported by an AWM-NSF Travel Grant in 2023.}
\thanks{A portion of this research was carried out as part of the American Institute of Mathematics (AIM) SQuaREs
program. The authors thank AIM for their support.}

\begin{document}

\author{Dana Barto\v{s}ov\'a}
\address{University of Florida, 1400 Stadium Road, Gainesville, FL 32601}
\email{dbartosova@ufl.edu}
\urladdr{https://people.clas.ufl.edu/dbartosova/}

\author{Lynn Scow}
\address{California State University, San Bernardino
	5500 University Parkway 
	San Bernardino, CA 92407} 
\email{lscow@csusb.edu} 
\urladdr{https://www.csusb.edu/profile/lynn.scow}

\begin{abstract}
We investigate the notion of a semi-retraction between two first order structures (in typically different signatures) that was introduced by the second author as a link between the Ramsey property and generalized indiscernible sequences. We further these connections between combinatorics and model theory, and look at semi-retractions through a new lens establishing transfers of the Ramsey property and finite Ramsey degrees under quite general conditions that are optimal as demonstrated by counterexamples. Finally, we compare semi-retractions to the category theoretic notion of a pre-adjunction. 
\end{abstract}

\maketitle

\section{Introduction}

Our current subject is finiteness of Ramsey degrees for embeddings of finitely generated structures and mechanisms that transfer this property between classes.
Given a first-order structure $\A$ in any signature, let $\K:=\age(\A)$ be the class of all finitely generated substructures of $\A$.  Given an element $A \in \K$, we may refer to all substructures of $\A$ isomorphic to $A$ as the \textit{$A$-substructures of $\A$}.  We say that $\K$ has the Ramsey property if it has a certain partition property as stated in Definition \ref{rpDef}: The class $\K$ has the \textit{Ramsey property (RP)} if for all $A, B \in \K$ and integers $k \geq 1$ there exists $C \in \K$ such that for any coloring of the $A$-substructures of $C$, there exists a copy $B'$ of $B$ in $C$ such that the coloring is constant on the $A$-substructures of $B'$.  There has been much recent work in structural Ramsey theory to understand the full landscape of classes of structures with the Ramsey property.  The Ne\v{s}et\v{r}il-R\"{o}dl and Abramson-Harrington theorems gave general classes of structures with the Ramsey property \cite{abha78,rone77}, such as linearly ordered graphs and hypergraphs.  Several years later, this work was extended to first order structures in signatures that are not purely relational \cite{hune19}, which is our present context.

%{\color{red} Define delta-n-indiscernble sequences?}
%%new paragraph
%Ramsey reduces it to checking all finite models of size $m$, which is a decidable procedure.
Ramsey, in his 1929 paper, described the \emph{Entscheidungsproblem} from the 1928 book by Hilbert and Ackermann, \emph{Grundz\"{u}ge der theoretischen Logik} as ``the problem of finding a regular
procedure to determine the truth or falsity of any given logical formula'' \cite{ra29}.  
%or, Principles of Theoretical Logic.
In an effort to address this problem for special cases of formulas, Ramsey proved what is now known as the infinite Ramsey theorem for sets, followed by the finite Ramsey theorem for sets (what we would call the Ramsey property for the class of finite sets, in the language of Definition \ref{rpDef}).  The main theorem in \cite{ra29} for universal sentences $F$ in a relational signature states that there is an integer $m$ depending on certain quantities derived from $F$, such that $F$ is consistent in a structure whose universe has $m$ or more members if and only if some logically equivalent sentence $P$ in disjunctive normal form contains a disjunct of a certain syntactic form.  Let $\Delta$ be a finite set of relation symbols and let $n$ be an integer.  It is the presence of $\Delta$-$n$-indiscernible sequences (see Definition \ref{deltanind}) of a certain size depending on $m$, guaranteed by Ramsey's combinatorial theorems, that guarantees the necessity of the condition in this main theorem.  For further implications of Ramsey's work, the interested reader is invited to consult \cite{sep-ramsey}.
%as a starting point for these investigations.
 Later, in \cite{EM56}, the infinite Ramsey theorem for sets is applied to obtain infinite $\Delta$-$n$-indiscernible sequences locally based on an infinite subset of a structure.  This notion was generalized in \cite{sh78} to $\I$-indexed indiscernible sets, which were the starting point for the investigations of the second author.
%

%Ramsey solves this problem for the case of sentences $\varphi$ in a relational signature using finitely many variables $x_1,\ldots,x_n$ such that the sentence is in a certain normal form; namely, a finite sequence of universal quantifiers precedes all other logical symbols in the sentence.  In order to obtain this result, 

An \textit{$\I$-indexed (generalized) indiscernible set} (see Definition \ref{genind}) is an $\I$-indexed set of same-length finite tuples $\ov{a}_i$ from some structure $\M$, for $i \in \I$, which sequence is homogeneous in a certain way: The structure $\M$ does not make any more distinctions between finite strings $\ov{a}_{\ov{\imath}}$, $\ov{a}_{\ov{\jmath}}$ of tuples than the atomic formulas in $\I$ do between $\ov{\imath}$ and $\ov{\jmath}$.  This tool gives streamlined proofs of the equivalence of certain dividing lines in classification theory in model theory, for example that a theory is unstable if and only if it has the independence property or the strict order property; or that a theory has the tree property if and only if it has the tree property of either the first or second kind \cite{sh78}. A property stated for a specific class of trees in \cite{dzsh04} was referred to as the \textit{modeling property} by the second author in \cite{sc12} and proved to be equivalent to $\age(\I)$ having the Ramsey property, under certain assumptions on $\I$.  In the Introduction to \cite{sc15}, this latter result was termed a ``dictionary'' theorem, 
%(and extended to a wider class of structures $\I$)
in the spirit that the theorem translates phrase \eqref{prop1} (stated in the language of model theory) into phrase \eqref{prop2} (stated in the language of partition theory) under mild conditions:

\begin{equation}\label{prop1}
\textrm{For the structure $\I$, $\I$-indexed indiscernible sets have the modeling property.}
\end{equation}

\begin{equation}\label{prop2}
\textrm{For the structure $\I$, $\age(\I)$ has the Ramsey property.} 
\end{equation}

It is arguable that the entanglement of logic and partition theory from the start (as in \cite{ra29}) renders the use of the term ``dictionary'' moot.  The second author appeals to recent history when the two properties described in \eqref{prop1} and \eqref{prop2} were not as well understood to be equivalent.  Though the Ramsey property was used extensively in \cite{sh78} in different contexts to prove the existence of $\I$-indexed indiscernible sets that model the definable relations on a pre-existing set of parameters (as was done in \cite{EM56}), it was not generally clear that these were equivalent properties, and indeed, that the failure of the Ramsey property would guarantee failure of the modeling property in specific contexts, to the detriment of certain model-theoretic investigations.  The outstanding hope is that manipulations on the model-theory side could yield new results in partition theory. 
 In fact, an argument in \cite{sc15} that leveraged the theory of generalized indiscernible sequences provided a new proof of the result in \cite{le73} that a certain class of trees in a functional signature has the Ramsey property. 
 In \cite{sc15}, the dictionary theorem was extended to include all cases when $\I$ is locally finite and ordered, modulo an additional property which was eliminated in \cite{sc21} as a result of conversations with the first author.  The complete Dictionary Theorem is stated in Theorem \ref{dictionary}.
  
In \cite{sc21}, the notion of \textit{semi-retraction} was introduced by the second author, which is a pair of maps between two structures $\A, \B$ in possibly different signatures (see Definition \ref{semir}).  This notion, together with the Dictionary Theorem, yielded a theorem detailing when a semi-retraction transfers the Ramsey property from $\B$ to $\A$, stated in Corollary \ref{transfer}.  Since the notion of semi-retraction is essentially algebraic, the first author suggested an investigation into a ``formula-free'' proof of the same result (one that does not specifically use the tools of first-order logic). This approach ultimately led to  Theorem \ref{thm_transembd}, Corollary \ref{cor_fmla_free} and Theorem \ref{fmla_free_relational}, which are sharper refinements of Corollary \ref{transfer}.
% which led to further insights into the dynamical consequences of the Ramsey property.  

The paper is organized as follows.  In Section \ref{prelim}, we outline our basic definitions and notational conventions as well as overview prior results.  
%and connections to pre-adjunctions.  
In Section \ref{examples}, we state a characterization of semi-retractions in terms of the Ramsey property under certain conditions, Theorem \ref{iff}, and give several examples and non-examples of semi-retractions.  
%In Section \ref{aut}, we indicate corollary results for the automorphism groups.  
In Section \ref{counterexample}, we show that the countable random graph and countable random $n$-regular hypergraphs for $n\geq 3$ are semi-retracts of the countable atomless Boolean algebra. 
In Section \ref{fmla}, we present our ``formula-free'' argument for how semi-retractions transfer the Ramsey property, Corollary \ref{cor_fmla_free}: If $\A$ is locally finite and the finitely-generated substructures of $\B$ are rigid, then if $\B$ has RP and $\A$ is a semi-retract of $\B$, $\A$ must have RP.   Moreover, Theorem \ref{fmla_free_relational} eliminates the assumption of rigidity in the case of relational structures.  The same technique yields a result on transfer of finite small Ramsey degrees in Corollary \ref{cor_fmla_free_degrees} and Theorem \ref{thm_transembd} (in the latter, rigidity may be omitted again).
In Section \ref{category}, we explore a categorical characterization of semi-retractions as well as the relationship to pre-adjunctions, which were introduced in \cite{mas18}. In section \ref{conclusion}, we discuss the relationship between semi-retractions and more commons constructions in model theory.

\section{Preliminaries}\label{prelim}

We present our basic notation around sequences.  
%Given a tuple $\ov{a} = (a_0,\ldots,a_{n-1})$ and a function $f$, by $f(\ov{a})$ we mean $(f(a_0),\ldots,f(a_{n-1}))$ and $\ran~\ov{a} :=\{a_i \mid i<n\}$.  
A tuple $\ov{a}$ is a finite sequence $(a_0,\ldots,a_{n-1})$ for some natural number $n$, and $|\ov{a}|=n$ is defined to be the length of the tuple.  Given a set $X$ and an integer $s \geq 1$, $X^s$ denotes the set of all $s$-tuples from $X$.
%Given a finite sequence $\ii$, $|\ii|$ denotes the integer length (domain) of the sequence.  
We define $\ran~\ov{a} :=\{a_0,\ldots,a_{n-1}\}$.
All tuples $\ov{a}$ are assumed to be finite unless said otherwise.
For two tuples $\ov{a}_1, \ov{a}_2$, by $\ov{a}_1 \subseteq \ov{a}_2$ we mean that $\ran~\ov{a}_1 \subseteq \ran~\ov{a}_2$ as sets.
For an integer $k \geq 0$, $k:=\{0,1,2,\ldots,k-1\}$.
For an integer $k \geq 1$ a \emph{$k$-coloring} of a set $X$ is any function $c: X \raw k$.  
%If $k$ is replaced by any other set of cardinality $k$, $c$ may still be referred to as a $k$-coloring
%, or even simply a \emph{finite coloring} if the size of $k$ is unimportant. 
%For a finite $n$-tuple $\ov{a}$ I sometimes just write $a$, and a
We denote the image $\{c(x) : x \in X\}$ by $c(X)$. 
%I guess not c''X
Given a function
%$f: X \raw Y$ and $\ii':=(\ii_k : k<s) \in {(X^n)}^s$, we define $f(\ii') = (f(\ii_k) : k<s)$.
%%Did we use this level of generality above?
%10.8
$f: X^n \raw Y$ and $\ii':=(i_k : k<s) \in X^s$, we define $f(\ii') = (f(i_k) : k<s)$.
Let $f:X\raw Y$ be an injective function and let $\psi:A\raw B$ be a function between two subsets of $X$. We will denote by $f(\psi):f(A)\raw f(B)$ the function defined by $f(\psi):=\{(f(a),f(\psi(a))):a\in A\}.$  Moreover, if $\psi$ is injective, $f(\psi)$ is also injective.
%9.29

%%%%%%LOGIC%%%%
%
A signature $L$ is a list of symbols that must be interpreted in any $L$-structure as either relations or functions of the specified arity.  The signature is \emph{relational} if it consists only of relation symbols, and \emph{functional} if it contains at least one function symbol.  We do not assume that signatures are either finite or relational, unless explicitly stated.  The cardinality of a set $S$ is denoted by $|S|$.  Given a structure $\A$, $\LA$ refers to the signature of $\A$, $|\A|$ refers to the underlying set of $\A$, and $||\A||$ to the cardinality of the set $|\A|.$ We denote by $\Th(\A)$ the theory of $\A$ -- the set of all $\LA$-sentences true in $\A$.  
Given $\ov{a}$ from $\A$, $\la \ov{a} \ra_\A$ denotes the substructure generated by $\ov{a}$ in $\A$ (it will be the closure of $\ov{a}$ under the $n$-ary function symbols of $\LA$ for all $n \geq 0$).  Given  a (possibly infinite) tuple $\ov{x}$ in 1-1 correspondence with an enumeration $\ov{a}$ of $\A$, by $\Diag_\A(\ov{x})$ we mean the set of all $R(x_{i_0},\ldots,x_{i_{n-1}})$ for relation symbols $R \in \LA$ such that $\A \vDash R(a_{i_0},\ldots,a_{i_{n-1}})$.
%
%Fix a structure $\A$, and 
For an integer $n \geq 1$, an $n$-ary $L$-formula is a first-order formula with free variables included in the list $x_0,\ldots,x_{n-1}$.
We say that a set $X \subseteq {|\A|}^n$ is \emph{definable} if there exists $n$ and an $n$-ary $\LA$-formula $\varphi(\ov{x})$ such that for all $\ov{a}$ from $\A$, $\A \vDash \varphi(\ov{a})$ if and only if $\ov{a} \in X$, and \emph{$0$-definable} if $\varphi(\ov{x})$ may be chosen without parameters from $\A$.
A set $X \subseteq {|\A|}^n$ is \emph{quantifier-free definable} if it is $0$-definable by way of a quantifier-free formula $\varphi(\ov{x})$.
%10.11
%
A structure $\A$ is \emph{$\omega$-homogeneous} if any partial isomorphism between finite subsets of $\A$ can be extended to an automorphism of $\A$.  For the basics of formulas, structures, and Fra\"iss\'e theory the reader is referred to \cite{ma02,ho93}.

For structures $\A, \B$, $\A \subseteq \B$ always denotes that $\A$ is a substructure of $\B$, in which case they are structures in the same signature.  The age of a structure $\A$, $\age(\A)$ is the set of all finitely-generated substructures of $\A$,  modulo $\LA$-isomorphism.  We will use Roman letters $A, B$ for finitely generated substructures of a given structure $\A$.
For structures $\A, \A'$, $\A \cong \A'$ means that the structures are isomorphic (and thus, they are in the same signature).  To emphasize the shared signature, we might write $\A \cong_L \A'$, where $L=\LA=L(\A')$.
We say that a structure $\A$ is \emph{rigid} if the only automorphism of $\A$ is the identity map.

\begin{rmk}\label{unique} If $\age(\B)$ consists of rigid elements, then for any $C, C' \in \age(\B)$, if $C \cong C'$, then this is witnessed by a unique isomorphism $\tau: C \raw C'$.
\end{rmk}

%\begin{proof} If $\tau_1, \tau_2 : C \raw C'$ are two isomorphisms,  then $\tau_2^{-1}\tau_1 : C \raw C$ is an automorphism of $C$,  but since $\Aut(C)$ is trivial, $\tau_2^{-1}\tau_1 = \id$,  i.e.  $\tau_1 = \tau_2$.  %Thus,  if $C \cong C'$, then there is a unique isomorphism from $C$ to $C'$, $\tau : C \raw C'$.\end{proof}

\noindent Fix a signature $L$, an $L$-structure $\A$, and an integer $n \geq 1$.
\begin{enumerate}  
\item\label{delta} Given a set $\Delta$ of $L$-formulas, 
%closed under variable substitutions, 
a \textit{$\Delta$-$n$-type (over $\emptyset$ in $\A$)} is a set of $n$-ary formulas from $\Delta$ that is consistent with $\Th(\A)$.  A \textit{complete $\Delta$-$n$-type} $\pi$ is a $\Delta$-$n$-type such that for every $n$-ary formula $\varphi$ from $\Delta$, either $\varphi \in \pi$ or $\neg \varphi \in \pi$.  If we drop the use of $n$, we mean $\Delta$-$n$-type for some $n$.
\item In the case that $\Delta$ is the set of all quantifier-free $L$-formulas, we call a $\Delta$-$n$-type in $\A$ a \textit{quantifier-free $n$-type in $\A$}.  In the case that $\Delta$ is the set of all $L$-formulas, we call a $\Delta$-$n$-type in $\A$ an \textit{$n$-type in $\A$}.  Such types are described as ``complete'' if they are complete $\Delta$-$n$-types for the appropriate $\Delta$.  
%\item Complete quantifier-free types and complete types are complete according to the instantiation of $\Delta$ in item \eqref{delta}.
\item An $n$-type $p$ over $\emptyset$ in $\A$ is \textit{realized (in $\A$)} if there exists $\ov{a} \in {|\A|}^n$ such that $\A \vDash \varphi(\ov{a})$ for all $\varphi \in p$.  
\item A structure $\A$ is \textit{$\kappa$-saturated} if for all subsets $S \subseteq |\A|$ such that $|S| < \kappa$, $\A$ realizes all types in $\A$ over $S$. We say that $\A$ is \textit{saturated}, if it is $||\A||$-saturated.
\item We define $S^\M_n(\emptyset)$ to be the space of all complete $n$-types over $\emptyset$ in $\A$ with the usual Stone topology, with basic open sets $[\psi] := \{p \in S^\M_n(\emptyset) \mid \psi \in p\}$.  
\item Given a tuple $\ov{a} \in {|\A|}^n$, $\tp_\Delta^\A(\ov{a})$ is the complete $\Delta$-$n$-type of $\ov{a}$ in $\A$.  For $\tp_\Delta^\A(\ov{a})$, we write $\qftp^\A(\ov{a})$ when $\Delta$ is the set of all quantifier-free formulas, and we write $\tp^\A(\ov{a})$ when $\Delta$ is the set of all formulas.
\item For same-length tuples $\ov{a}, \ov{a}'$ from $\A$, we will use $\ov{a} \sim_\A \ov{a}'$ to mean that $\qftp^\A(\ov{a})=\qftp^\A(\ov{a}')$, $\ov{a} \equiv_\Delta^\A \ov{a}'$ to mean that $\tp_\Delta^\A(\ov{a})=\tp_\Delta^\A(\ov{a}')$, and $\ov{a} \equiv^\A \ov{a}'$ to mean that $\tp^\A(\ov{a})=\tp^\A(\ov{a}')$.
%and with no repetitions.)
\end{enumerate}

\begin{rmk}\label{shelah}
If a structure $\A$ is saturated then $\A$ realizes all types $p \in S^\A_\kappa(\emptyset)$ such that $\kappa \leq ||\A||$ (see Lemma 1.12 in \cite{sh78}).
\end{rmk}

%%%%%%%above is from other note.  combine.
\subsection{Structural Ramsey theory and topological dynamics}
We start this section with some standard definitions from structural Ramsey theory, see \citep[Introduction of part (D)]{kpt05},\cite{nvt13},\cite{ne05}.  Given $L$-structures $A, B$ we define ${B \choose A}$ to be the set of all substructures $A' \subseteq B$ such that $A' \cong A$ and $\emb(A,B)$ to be the set of all $L$-embeddings $f:A \rightarrow B$.  Given an $L$-embedding $h: B \rightarrow C$, we define $h \circ \emb(A,B) = \{h \circ f : f \in \emb(A,B) \} \subseteq \emb(A,C)$.

Following \cite{draganKPT}, we define two types of Erd\H{o}s-Rado partition arrow
%incorporated in introduction
\begin{dfn}
Given a signature $L$, $L$-structures $A, B, \M$, and integers $k, d \geq 1$, the notation 
$$\M \longrightarrow (B)^A_{r,d}$$ 
denotes that for all $r$-colorings $c: {\M \choose A} \rightarrow r$, there exists $B' \subseteq \M$, $B' \cong B$, such that $\left|c({B' \choose A})\right| \leq d$.
		
We say that the structure $B'$ above is \emph{$\leq d$-chromatic (for the coloring $c$ on copies of A)}.

If $d=1$, it will be dropped in the notation, and we will write
$$\M \longrightarrow (B)^A_{r}.$$ 

Moreover,
$$\M \xlongrightarrow{e} (B)^A_{r,d}$$ 
denotes that for all $r$-colorings $c: \emb(A,\M) \rightarrow r$, there exists $h \in \emb(B,\M)$ such that $|c(h \circ \emb(A,B))| \leq d$.

If $d=1$, it will be dropped in the notation, and we will write
$$\M \xlongrightarrow{e} (B)^A_{r}.$$ 
\end{dfn}

\begin{dfn} Let $\K$ be a class of finitely-generated $L$-structures, for some signature $L$, and let $A, B \in \K$.  
We say that $(A, B)$ is a \emph{Ramsey duo for $\K$} if for all integers $r \geq 2$ there exists $C \in \K$ such that 
$$C \longrightarrow (B)^A_{r} .$$
%for any $k$-coloring $c$ of ${C \choose A}$, there is $B'$ $\in {C \choose B}$ such that for any $A', A'' \in {B' \choose A}$, $c(A') = c(A'')$.  

We say that $B'$ is a copy of $B$ \emph{homogeneous for $c$ (on copies of $A$)}.
\end{dfn}

\begin{dfn}\label{rpDef}
We say that $\K$ has \emph{the Ramsey property (RP)} if for all $A, B \in \K$, $(A, B)$ is a Ramsey duo for $\K$.
\end{dfn}

%We also list examples of classes, mostly listed in \cite{kpt05,ne05}, as they fall within this characterization.

\begin{ex}\label{ex_RP}  The following classes have RP:
\begin{enumerate}
	\item All finite sets in $L = \emptyset$ (\cite{ra29}).
	\item All finite linear orders in $L= \{<\}$ (\cite{ra29}).
\item All finite simple graphs with no loops with an ordering on the vertices in $L = \{R,<\}$ (\cite{abha78},\cite{rone77}).
\item All finite $n$-regular hypergraphs with linear orders in $L=\{R,<\}$, where $R$ is $n$-ary (\cite{abha78},\cite{rone77}).
\item Convexly ordered finite equivalence relations in $L= \{E,<\}$ (known, see discussion after Corollary 6.8 in \cite{kpt05}).
\item Finite Boolean algebras in $L=\{\vee,\wedge,\neg,\mathbf{0},\mathbf{1}\}$ (\cite{graro71}).
\end{enumerate}
\end{ex}

\begin{ex}\label{ex_notRP}  The following classes do not have RP:
\begin{enumerate}
\item All finite simple graphs with no loops (\cite{rone75}).
\item Finite equivalence relations with any ordering on points in $L = \{E,<\}$ (Theorem 6.4 in \cite{kpt05}).
\item Partial orders with any linear ordering on points in $L = \{<,\prec\}$ (\cite{sok111}).
\end{enumerate}
\end{ex}

As we can see in examples above, the class of finite graphs does not have the Ramsey property, but its expansion by linear orders does. This phenomenon leads to the notion of a Ramsey degree.

\begin{dfn}
	Let $\K$ be a class of $L$-structures and let $A\in \K$. We say that $A$ has \emph{finite Ramsey degree in $\K$} if there is an integer $d \geq 1$ such that for every $B\in\K$ and every $r\geq 2$, there is $C\in \K$  such that $C \longrightarrow (B)^A_{r,d}$.

 We define $d(A,\mathcal{K})$ to be the least such integer $d$, if it exists, and otherwise define $d(A,\mathcal{K})=\infty$.  If $d(A,\K)$ is finite, it is the \emph{Ramsey degree of $A$ in $\K$}.
  % Minimal such $d$ is (if it exists) the \emph{Ramsey degree} of $A$ in $\K$ and is denoted by $d(A,\mathcal{K})$.
	\end{dfn}

 \begin{obs} A class $\K$ has RP if and only if $d(A,\K)=1$ for all $A \in \K$.
 \end{obs}

We can make a related definition for an infinite structure $\M$.

\begin{dfn}
Given an $L$-structure $\M$ and a finitely-generated substructure $A \subseteq \M$, we say that $A$ has \emph{finite small Ramsey degree in $\M$} if for some integer $d \geq 1$, for all finitely-generated structures $B \subseteq \M$, for all integers $r \geq 2$, 
$$\M \longrightarrow (B)^A_{r,d}.$$ 

If $A$ has finite small Ramsey degree in $\M$, then we define $d(A,\M)$ to be the least integer $d$ such that for all finitely generated structures $B \subseteq \M$, for all integers $r \geq 2$, $\M \longrightarrow (B)^A_{r,d}$.  

If $A$ does not have finite small Ramsey degree in $\M$, we define $d(A,\M) = \infty$.

We say that $(A,B)$ is a \emph{Ramsey duo for $\M$} if for all integers $r \geq 2$, 
$$\M \rightarrow (B)^A_{r}$$
\end{dfn}

\begin{prop}\label{compactness_duo} Fix a signature $L$ and a locally finite $L$-structure $M$.  Let $\K:=\age(\M)$.  Then, for any finite substructures $A, B \subseteq \M$,
$(A,B)$ is a Ramsey duo for $\K$ if and only if $(A,B)$ is a Ramsey duo for $\M$.
\end{prop}

A proof for  Proposition \ref{compactness_duo} is straightforward and provided in the Appendix.

%%this was deleted

Given the ability to do calculations in a countably infinite structure, the following is common usage:

\begin{dfn}
We say that \emph{$\A$ has RP} if $\age(\A)$ has RP.  
\end{dfn}

%\vspace{.1in}

Striking connections between dynamics of an automorphism group of an $\omega$-homogeneous structure $\A$ and Ramsey degrees of $\age(\A)$ were established in \cite{kpt05}. The work of these authors and others shows that this relationship is best explained in terms of the Ramsey properties of embeddings rather than substructures.
%For structures $A, B$   in the same signature, we denote by $\emb(A,B)$ the set of all embeddings of $A$ into $B$. A structure $\A$ is called \emph{$\omega$-homogeneous} if every finite partial isomorphisms of $\A$ extends to an automorphism of $\A$.

\begin{dfn}
	Let $\K$ be a class of $L$-structures and let $A\in \K$. We say that $A$ has \emph{finite Ramsey degree for embeddings in $\K$} if there is an integer $d \geq 1$ such that for every $B\in\K$ and every $r\geq 2$ there is $C\in\K$ such that for every coloring $c:\emb(A,C)\raw\{0,1,\ldots,r-1\},$ there is $h\in \emb(B,C)$ such that $c$ on $h\circ\emb(A,B)$ takes at most $d$ colors.
%deleted: =\{h\circ f:f\in\emb(A,B)\}

  We define $d_e(A,\mathcal{K})$ to be the least such integer $d$, if it exists, and otherwise define $d_e(A,\mathcal{K})=\infty$.  If $d_e(A,\K)$ is finite, it is the \emph{Ramsey degree for embeddings of $A$ in $\K$}.

If $d_e(A,\K)=1$ for all $A \in \K$, then we say that $\K$ has the \emph{Ramsey property for embeddings}.
%Lynn 04-15
\end{dfn}

\begin{dfn}
Given an $L$-structure $\M$ and a finitely-generated substructure $A \subseteq \M$, say that $A$ has \emph{finite small Ramsey degree for embeddings in $\M$} if for some integer $d \geq 1$, for all finitely-generated structures $B \subseteq \M$, for all integers $r \geq 2$, 
$$\M \xlongrightarrow{e} (B)^A_{r,d}.$$ 

  We define $d_e(A,\M)$ to be the least such integer $d$, if it exists, and otherwise define $d_e(A,\M)=\infty$.  If $d_e(A,\M)$ is finite, it is the \emph{Ramsey degree for embeddings of $A$ in $\M$}.
  
%If $A$ has finite small Ramsey degree for embeddings in $\M$, then we define $d_e(A,\M)$ to be the least integer $d$ such that for all finitely generated structures $B \subseteq \M$, for all integers $r \geq 2$, $\M \xlongrightarrow{e} (B)^A_{r,d}$.  

%If $A$ does not have finite small Ramsey degree for embeddings in $\M$, we define $d_e(A,\M) = \infty$.

%We say that $(A,B)$ is a \emph{Ramsey duo for embeddings in $\M$} if for all integers $r \geq 2$, $$\M \rightarrow (B)^A_{r}$$
\end{dfn}

It is well known that the Ramsey property for a class $\K$ that is an age of finite structures can be understood as a property of an infinite structure $\M$ with $\age(\M) = \K$. This is proved for a countable age of finite structures in Proposition 3 of \cite{nvt13} using ultrafilters, and the generalization to Ramsey degrees is proved in the more general setting of a category $\mathbb{C}$ with a distinguished subcategory $\mathbb{C}_{\textrm{fin}}$ satisfying certain assumptions in Lemma 3.4 of \cite{draganKPT}.
In the Appendix we provide a model-theoretic argument for Lemma \ref{compactness_emb_degrees} by compactness as an alternative approach.

\begin{lem}\label{compactness_emb_degrees} Fix a signature $L$ and a locally finite $L$-structure $\M$.  Let $\K:=\age(\M)$.  Then, for any finite substructure $A \subseteq \M$,
\begin{enumerate}
    \item $d_e(A,\K) = d_e(A,\M)$, and
    \item $d(A,\K) = d(A,\M)$. 
\end{enumerate}
%(Thus, $A$ has finite Ramsey degree for embeddings in $\K$ if and only if $A$ has finite small Ramsey degree for embeddings in $\M$.)
\end{lem}

The relationship between Ramsey degrees for structures and Ramsey degrees for embeddings has a long history, dating back to work in \cite{abha78}, \cite{rone77}, \cite{fou99}.  This history is cataloged in the introduction to Section 10 of \cite{kpt05}, in which the case of $\K$ having an ordered expansion with RP is worked out in detail.  A more general result in a category theory context is offered by Proposition 3.1 of \cite{draganKPT}.

\begin{prop}[\cite{kpt05,draganKPT}]\label{TwoDegrees} Given a signature $L$ and a class $\K$ of finite $L$-structures and $A \in \K$, if $d_e(A,\K)$ is finite, then so is $d(A,\K)$ and
$$d_e(A,\K) = |\Aut(A)| \cdot d(A,\K).$$
\end{prop}

Below we state the famous Kechris-Pestov-Todor\v{c}evi\'c correspondence from \cite{kpt05} between the Ramsey property of $\age(\A)$ for a (countable) $\omega$-homogeneous structure $\A$ and the fixed point on compacta property of $\Aut(\A).$ We start by introducing the necessary notions from topological dynamics.

	Let $G$ be a topological group and $X$ a compact Hausdorff space. A continuous function $\alpha:G\times X\to X$ is a \emph{$G$-flow} if 
	\begin{enumerate}
		\item $\alpha(e,x)=x$ for any $x\in X$ and $e$ the neutral element of $G$,
		\item $\alpha(gh,x)=\alpha(g,\alpha(h,x))$ for every $g,h\in G$ and $x\in X$.
	\end{enumerate}
	In, other words, a flow is a continuous group action, where continuity is considered with respect to the product topology. 
	We typically write $gx$ in place of $\alpha(g,x).$
	A $G$-flow on $X$ is \emph{minimal} if $X$ does not contain a non-empty proper closed $G$-invariant subset. A \emph{homomorphism} between $G$-flows $X$ and $Y$ is a $G$-equivariant continuous map $\phi:X\to Y,$ i.e., for every $g\in G$ and $x\in X,$ we have $\phi(gx)=g\phi(x).$ If $\phi$ is onto, we say that $Y$ is a \emph{quotient} of $X$ and if $\phi$ is bijective, it is called an \emph{isomorphism}. Ellis showed that up to isomorphism, for every topological group $G$, there is a unique \emph{universal minimal flow}, $M(G)$, that is, a minimal $G$-flow which has every minimal $G$-flow as a quotient (see \cite{ellis60} for discrete and \cite{ellis69} for arbitrary groups). We call $G$ \emph{extremely amenable} if every $G$-flow $X$ has a fixed point --  a point $x_0\in X$ such that $gx_0=x_0$ for every $g\in G.$ It immediately follows that $G$ is extremely amenable if and only if $M(G)$ is a single point (and thus every minimal $G$-flow is a single point).
	
	\begin{thm}[\cite{kpt05},\cite{bar12} for uncountable structures]\label{ea}
		Let $\A$ be an $\omega$-homogeneous structure. The following are equivalent.
		\begin{enumerate}
			\item The group $\Aut(\A)$ is extremely amenable.
			\item The class $\age(\A)$ satisfies the Ramsey property and consists of rigid elements -- structures whose automorphism group is trivial.
		\end{enumerate}
	\end{thm}
 In fact a consequence of Proposition \ref{TwoDegrees} was observed as Proposition 2.3 in \cite{MaSc17} stated in a more general category theory context, namely that a class $\K$ of finite structures has the Ramsey property for embeddings if and only if $\K$ has the Ramsey property for substructures and every member of $\K$ is rigid.
 Thus item (2) in Theorem \ref{ea} can be  replaced by 
  \begin{enumerate}
  	\item[(2)']
  $\age(\A)$ satisfies the Ramsey property for embeddings. 
\end{enumerate}

In \cite{kpt05}, the authors computed a number of universal minimal flows of automorphism groups of countable $\omega$-homogeneous structures whose ages have finite Ramsey degrees. In fact, Zucker later proved in \cite{zu16} that finite Ramsey degrees are equivalent to the universal minimal flow being metrizable, as stated precisely below.

\begin{thm}[\cite{kpt05}, \cite{zu16}]\label{zucker}
Let $\A$ be a countable $\omega$-homogeneous structure. The following are equivalent.
		\begin{enumerate}
			\item The universal minimal flow $M(\Aut(\A))$ is metrizable.
			\item The class $\age(\A)$ has finite Ramsey degrees for embeddings.
		\end{enumerate}
\end{thm}

%There is a natural relationship between the Ramsey degree for substructures and the Ramsey degree for embeddings. 
%

\subsection{Semi-retractions}
%%%

In this section, we will review the notions that informed the proof of the Ramsey transfer result Corollary \ref{transfer}.  These notions come from model theory, but we will see later that the more general approach in the current paper allows us to drop some of the assumptions originally thought to be necessary in Corollary \ref{transfer}, as evidenced by Corollary \ref{cor_fmla_free}.
%The following two definitions were introduced in \cite{sc21}.

\begin{dfn}[\cite{sc21}]\label{def_semir}  Given any structures $\A, \B$, not necessarily in the same signature, we say that an injection $h: \A \raw \B$ is
\begin{itemize}
\item[(i)]
 \emph{quantifier-free type-respecting (qftp-respecting)} if for all finite, same-length tuples $\ii, \jj$ from $\A$,
$$\ii \sim_\A \jj \Rightarrow h(\ii) \sim_\B h(\jj) .$$
\item[(ii)]
 \emph{quantifier-free type-preserving (qftp-preserving)} if $\A, \B$ are structures in the same signature and $\qftp^\A(\ii) = \qftp^\B(h(\ii))$ (thus, it is also qftp-respecting).
\end{itemize}
\end{dfn}

\begin{dfn}[\cite{sc21}]\label{semir} Let $\A$, $\B$ be any structures.  We say that $\A$ is a \emph{semi-retract of $\B$ (via $(g,f)$)} if
 
\begin{enumerate}
\item\label{semi1} there exist qftp-respecting injections $\A \xrightarrow{g} \B  \xrightarrow{f} \A$,
\item\label{semi2} such that
$\A \xrightarrow{fg} \A$
is an embedding (equivalently, is qftp-preserving).
\end{enumerate}
We refer to the pair $(g,f)$ as the \emph{semi-retraction between $\A$ and $\B$}.
%March 05
We will refer to property \eqref{semi1} in this Definition as the \textit{qftp-respecting property of semi-retractions} and property \eqref{semi2} as the \textit{composition property of semi-retractions}.
\end{dfn}

\begin{obs}\label{csb} If $\A$ is a semi-retract of $\B$, then $||\A||=||\B||$, by the Schr\"{o}der-Bernstein theorem.
\end{obs}
%deleted: April 3.

\subsection{The modeling property}

The first use of Definition \ref{deltanind} was in \cite{EM56} to construct models with many automorphisms.  The presentation of Definition \ref{deltanind} that we adopt in this paper can be found in \cite{sh78}.

\begin{dfn}[\cite{EM56},\cite{sh78}]\label{deltanind} Given a structure $\M$, a set $\Delta$ of $L(\M)$-formulas, and an integer $n \geq 1$, a
\textit{$\Delta$-$n$-indiscernible} sequence is
a sequence of finite $l$-tuples $\ov{a}_i$ from $\M$ for some integer $l \geq 1$ indexed by some linear order $\I = (|\I|,<)$ such that for all increasing $n$-tuples $\ii,\jj$ from $\I$,
$$\ov{a}_{\ii} \equiv_\Delta^\M \ov{a}_{\jj} .$$
\end{dfn}

In the study of classification theory in model theory there has been significant use of a generalization of this notion named ``$\I$-indexed indiscernible sets'' in \cite{sh78} which we will define as follows.

%\citep[Definition 2.1]{sc15}
%%DEF no assump Def 2.1 2015
\begin{dfn}[\cite{sh78}]\label{genind} Fix a structure $\I$, an integer $l \geq 1$, and $l$-tuples $\ov{a}_i$ from some structure $\M$, for all $i \in \I$.  We say that $(\ov{a}_i \mid i \in \I)$ is an \emph{$\I$-indexed indiscernible set} if for any integer $n \geq 1$, for all $n$-tuples $\ii,\jj$ from $\I$,
$$\ii \sim_\I \jj \Rightarrow  \ov{a}_{\ii} \equiv^\M \ov{a}_{\jj} .$$

We say that $(\ov{a}_i \mid i \in \I)$ is an \emph{$\I$-indexed indiscernible sequence} if $\I$ is an ordered structure, and additionally a \emph{generalized indiscernible sequence} when $\I$ is clear from context.
\end{dfn}

We review some definitions and basic results from \cite{sc21}, where the second author gave a more complete proof that the modeling property is a model-theoretic analogue of the Ramsey property.  
%In our applications below, the infinite structure $\I$ is typically replaced by $\A$. 

%We repeat definitions from \cite{sc15} as Definition \ref{em} and Definition \ref{mpd}.

%%DEF no assump  Def 2.6 2015
\begin{dfn}[\cite{sc15}]\label{em}  Given an integer $l \geq 1$, an $L'$-structure $\I$, an $L$-structure $\M$ and an $\I$-indexed set of $l$-tuples from $\M$, $X = (\ov{a}_i \mid i \in \I)$, we define the \emph{EM-type of $X$ ($\emtp(X)$)} to be a syntactic type in variables $(\ov{x}_i \mid i \in \I)$, where $|\ov{x}_i| = l$ for each $i \in \I$, as follows:
$$\emtp(X) = \{ \psi(\ov{x}_{i_0},\ldots,\ov{x}_{i_{n-1}}) \mid \psi \in L, \ov{\imath} \in {|\I|}^n \textrm{~and~} (\forall  \ov{\jmath} \in {|\I|}^n)( \ov{\jmath} \sim_\I \ov{\imath} \Rightarrow \M \vDash \psi(\ov{a}_{j_0},\ldots,\ov{a}_{j_{n-1}})) \}$$
\end{dfn}

%Proposition \ref{useful} is a useful equivalence which follows directly from Definition \ref{em} (see Proposition 2 of \cite{sc15} for more details):

%%PROP no assump
\begin{prop}[Proposition 2 of \cite{sc15}]\label{useful}Given an $L'$-structure $\I$ and an $L$-structure $\M$,
fix sets of $l$-tuples from $\M$ indexed by $\I$,
$X = (\ov{a}_i \mid i \in \I)$ and
$Y = (\ov{b}_i \mid i \in \I)$.
$Y \vDash \emtp(X)$ if and only if for any integer $n \geq 1$, for all complete quantifier-free $n$-types $\eta$ in $\I$ and all $n \cdot l$-ary formulas $\varphi \in L$, if we have the \emph{rule}
$$(\forall \jj) ( \I \vDash \eta(\jj) \Rightarrow \M \vDash \varphi(\ov{a}_{\jj}))$$
then we have the rule
$$(\forall \jj) ( \I \vDash \eta(\jj) \Rightarrow \M \vDash \varphi(\ov{b}_{\jj}))$$
\end{prop}

%%%%
\begin{dfn}[\cite{sc15}]\label{locbased}  Fix  sequences of parameters $X = (\ov{a}_i \mid i \in \A)$, $Y = (\ov{b}_i \mid i \in \A)$, where $\ov{a}_i, \ov{b}_i$ are from some $L$-structure $\M$.

We say \emph{$Y$ is locally based on $X$} if for any finite set of $L$-formulas, $\Delta$, and for any finite tuple $(j_0, \ldots, j_{n-1})$ from $\A$, there exists a tuple $(i_0, \ldots, i_{n-1})$ from $\A$ such that 

$$\ov{\jmath} \sim_\A \ov{\imath}$$

and

$$\ov{b}_{\jj} \equiv^\M_\Delta \ov{a}_{\ii}$$

\noindent Where $Y$ and $X$ are understood from context, this property will be referred to as \emph{local basedness}.
\end{dfn}

\noindent We give a proof sketch to illustrate the idea behind Proposition \ref{twodefs}.
\begin{prop}[Proposition 2 in \cite{sc15}]\label{twodefs}  Fix  sequences of parameters $X = (\ov{a}_i \mid i \in \I)$, $Y = (\ov{b}_i \mid i \in \I)$, where $\ov{a}_i, \ov{b}_i$ are $l$-tuples from some $L$-structure $\M$, for some integer $l \geq 1$. We have that  $Y$ is locally based on $X$ if and only if $Y \vDash \emtp(X)$.
\end{prop}

\begin{proof}  Suppose $Y \vDash \emtp(X)$.  To show local basedness, fix a finite set of $L$-formulas $\Delta$ and a finite tuple $\ov{\jmath}$ from $\I$ with complete quantifier-free type $\eta$ in $\I$.  Let $\varphi$ be a conjunction of all the formulas in the finite $\Delta$-type of $\ov{b}_{\ov{\jmath}}$.  Suppose, for contradiction, there is no $\ov{\imath} \sim_\I \ov{\jmath}$ such that $\ov{a}_{\ov{\imath}} \equiv^\M_\Delta  \ov{b}_{\jmath}$.   Then
$$(\forall \ii) ( \I \vDash \eta(\ii) \Rightarrow \M \vDash \neg \varphi(\ov{a}_{\ii})).$$
Since $Y \vDash \emtp(X)$, we must have that
$$(\forall \jj) ( \I \vDash \eta(\jj) \Rightarrow \M \vDash \neg \varphi(\ov{b}_{\jj})),$$
which contradicts the $\Delta$-type of $\ov{b}_{\jj}$.

Suppose $Y$ is locally based on $X$.  To show $Y \vDash \emtp(X)$, consider a rule from $\emtp(X)$:
$$(\forall \ii) ( \I \vDash \eta(\ii) \Rightarrow \M \vDash \varphi(\ov{a}_{\ii})).$$
Fix any $\jj$ from $\I$ such that $\I \vDash \eta(\jj)$.  By local basedness, there is $\ii \sim_\I \jj$ such that $  \ov{b}_{\jj} \equiv_{\{\varphi\}} \ov{a}_{\ii}$.  By the rule for $\eta$, $\M \vDash \varphi(\ov{a}_{\ii})$.  Thus we have that $\M \vDash \varphi(\ov{b}_{\jj})$, as well.  And so we have proved the rule
$$(\forall \jj) ( \I \vDash \eta(\jj) \Rightarrow \M \vDash \varphi(\ov{b}_{\jj})).$$
Since this is true for any rule, $Y \vDash \emtp(X)$.
\end{proof}

%If $\I$ is ordered by a 0-definable relation in $\I$, it is trivial to produce $\I$-indexed indiscernible sets, by Ramsey's theorem for finite sequences.  The following property guarantees that we can produce $\I$-indexed indiscernible sets that witness additional structure.

%%DEF no assump, equiv to Def 3.1 by Prop 2(3)(4) in 2015
\begin{dfn}[\cite{sc15}]\label{mpd} 
Given a structure $\I$, we say that $\I$-indexed indiscernible sets have the \emph{modeling property} if for any  integer $l \geq 1$, any $||\I||^+$-saturated structure $\M$, and any $\I$-indexed set of $l$-tuples from $\M$ 
$$X = (\ov{a}_i \mid i \in \I) ,$$
there exists an $\I$-indexed indiscernible set of $l$-tuples from $\M$
$$Y = (\ov{b}_i \mid i \in \I),$$
such that $Y \vDash \emtp(X)$ \textit{(equivalently, $Y$ is locally based on $X$, by Proposition \ref{twodefs}.)}

%We say that $Y$ is \emph{locally based} on $X$.
\end{dfn}

\begin{rmk}\label{rmk_saturation}  In fact, it suffices to require that $\M$ in Definition \ref{mpd} be $||\I||$-saturated, since the type describing $Y$ has $||\I||$ variables and no parameters from $\M$, as noted in Remark \ref{shelah}.  Previously, in \cite{sc15}, the bound was required to be $||\I||^+$ not $||\I||$.
\end{rmk}

\begin{thm}[Dictionary Theorem,\cite{sc12,sc15,sc21}]\label{dictionary} Suppose that $\I$ is a locally finite ordered structure.
Then $\I$-indexed indiscernible sequences have the modeling property if and only if $\age(\I)$ has RP.
\end{thm}

%\bigskip

Theorem \ref{dictionary} fails when we drop order:

\begin{ex}\label{noorder} Let $\I=(\Nn,=)$ and note that $\age(\I)$ has RP by Example \ref{ex_RP}.  If we take an $\I$-indexed set in $\M:=(\Nn,<)$, $X = (i \mid i \in \I)$, then there is no $\I$-indexed indiscernible set in any extension $\M' \succeq \M$ locally based on $X$.  Such a set would need to have $\tp^{\M'}(i,j)=\tp^{\M'}(j,i)$ for $i\neq j\in \Nn$, which is not possible.  
\end{ex}
%also possible to have a constant for every element.

Example \ref{noorder} illustrates why rigidity has been important in applications of structural Ramsey theory to generalized indiscernible sequences in model theory.  Consider an injection $f: \I \rightarrow \M$ and parameters $(a_i : i \in \I)$ such that $a_i = f(i)$, for all $i \in \I$.  Given a finite substructure $A \subseteq \I$ of size $n$, the injection $f$ induces a (possibly infinite) coloring on tuples $(a_0,\ldots,a_{n})$ such that $\ran (\ov{a}) \cong A$, where $\tp^\M (f(\ov{a}))$ is the color of $\ov{a}$.  Thus, if finitely-generated substructures of $\M$ are rigid, then we will not have an $\I$-indexed indiscernible set locally based on $(a_i : i \in \I)$ if $A$ has a nontrivial automorphism.  Note that the modeling property (Definition \ref{mpd}) is a universal statement about all structures $\M$, and so it would immediately fail for $\I$-indexed indiscernible sets if $\I$ is not rigid.  The reason for the modeling property to be a universal property is so that $\I$ can function as a tool in classification theory to compare all theories, even those theories whose models have rigid finitely-generated substructures.  In fact, theories whose models are linearly ordered by a formula in the language play an important role in classification theory as they are unstable and have the strict order property.

Theorem \ref{dictionary} fails when we drop local finiteness:

\begin{ex} Let $\I = (\Z,p,s,<)$ be the structure on $\Z$ with the usual order $<$ and where $p, s$ are unary function symbols interpreted as ``predecessor'' and ``successor'', respectively.  The only possible finitely generated substructure of $\Z$ is the whole structure.  Since $||\age(\I)||=1$, the class trivially has RP.  However, we will show that $\I$-indexed indiscernibles do not have the modeling property, showing the essentialness of the assumption that $\I$ be locally finite in Theorem \ref{dictionary}. Let $\M$ be the \fr~limit of finite convexly ordered equivalence classes in signature $\{E,\prec\}$.  Let $X=(a_i \mid i \in \Z)$ be such that all $a_i$ for $i$ odd are in one $E$-class that we call $\texttt{Odd}$ and all $a_j$ for $j$ even are in a separate $E$-class that we call $\texttt{Even}$.  Moreover, let $i < j \Rightarrow a_i \prec a_j$, and $\texttt{Odd}<\texttt{Even}$ in $\M$.  Within this example, we use $\sim$ to denote $E$-equivalence in the following visualization, where elements are listed in $\prec$-increasing order in $\M$:
$$\ldots a_1 \sim a_3 \sim a_5 \ldots \nsim \ldots a_2 \sim a_4 \sim a_6 \ldots$$
The type $\emtp(X)$ requires that whenever $j=s(s(i))$, $x_i \prec x_j$ and $E(x_i,x_j)$.  However, a decision is not made about $x_i \prec x_j$ when $j=s(i)$, since sometimes $a_i \prec a_j$ and sometimes $a_j \prec a_i$, when $j=s(i)$, though $a_i, a_j$ are always $E$-inequivalent when $j=s(i)$.  Suppose there is an $\I$-indexed indiscernible set locally based on $X$.  If this indiscernible set chooses the rule that $b_i \prec b_j$, whenever $j=s(i)$, then $\M$ would admit equivalence classes that are not convexly ordered, e.g. with $b_1 \sim b_3$ but
$$\ldots b_1 \nsim b_2 \nsim b_3 \ldots$$
 a contradiction.  The alternative, choosing $b_i \succ b_{s(i)}$, is incompatible with $b_i \prec b_{s(s(i))}$.
\end{ex}

\begin{thm}[\cite{sc21}]\label{thm1}  Let $\A$ and $\B$ be any 
%don't need countably infinite   Fri July 24
structures.  Suppose that $\A$ is a semi-retract of $\B$.  Furthermore, suppose that $\B$-indexed indiscernible sets have the modeling property.  
Then $\A$-indexed indiscernible sets have the modeling property.
\end{thm}

\begin{cor}[\cite{sc21}]\label{transfer} Let $\A$ and $\B$ be locally finite ordered structures.  Suppose that $\A$ is a semi-retract of $\B$ and $\B$ has RP.  Then $\A$ has RP.
\end{cor}

\begin{proof}
    This follows by Theorem \ref{thm1} and Theorem \ref{dictionary}.
\end{proof}

%%%%%%%%%%
\section{Semi-retractions, Reducts and Examples}\label{examples}

In this section, we consider basic examples of semi-retractions and give a characterization of semi-retractions in Theorem \ref{iff} under certain assumptions.

\begin{dfn} We say that $\A$ is a \emph{quantifier-free reduct} of $\B$ if $|\A|=|\B|=M$ and $\sim_\B$ refines $\sim_\A$ on $M$, i.e. for all finite same-length tuples $\ii, \jj$ from $|\A|$, $\ii \sim_\B \jj \Rightarrow \ii \sim_\A \jj$.
%We say that $A$ is \emph{quantifier-free $\aleph_0$-categorical
\end{dfn}

\begin{obs}\label{obs_reduct} Note that $\A$ is a quantifier-free reduct of $\B$ if and only if $|\A|=|\B|$ and the identity map $id : \B \raw \A$ is qftp-respecting.
\end{obs}

%%%%%%%%
\begin{ex} For two structures $\A, \B$ such that $|\A|=|\B|$, $\A$ is a quantifier-free reduct of $\B$ if any of the following hold.
\begin{enumerate}
\item The signature $\LA$ is contained in the signature $\LB$.
%\item there exists a qftp-respecting injection $f: \B \raw \A$ that is surjective onto $\A$
\item Every atomic formula of $\A$ with no parameters is equivalent to a quantifier-free formula of $\B$ with no parameters.
%changed out basic relation
\end{enumerate}
\end{ex}

\begin{dfn}  We say that $\A$ and $\B$ are \emph{quantifier-free interdefinable} if $|\A|=|\B|$ and each of $\A$, $\B$ is a quantifier-free reduct of the other.
\end{dfn}

\begin{rmk}\label{interdef} If $|\A|=|\B|$, then $\A$ and $\B$ are quantifier-free interdefinable if and only if the pair of identity maps between $\A$ and $\B$ give a semi-retraction (in either order).
\end{rmk}

%\begin{dfn} We say that $A$ is $\ldots ?? \ldots$ if $A$ has finitely many quantifier-free $n$-types, for each $n \geq 1$.\end{dfn}

%\textit{What is a better term to use?  Dugald suggests this could be phrased using Caroline Terry's notion of..speed?}

%%%%%%%%
%\begin{rmk} $\A$ has finitely many quantifier-free $n$-types, for each $n \geq 1$ if any of the following:
%\begin{enumerate}
%\item $\A$ has a finite relational signature
%\item $\A$ (meaning, $\Th(\A)$) is $\aleph_0$-categorical
%\item conditions to get an $\aleph_0$-categorical \fr-limit of a \fr-class: uniformly locally finite in a finite signature
%\end{enumerate}
%\end{rmk}

%work on
%I think there is an expansion of this result to when tps in A are a union of finitely many types in B and there is enough crosscutting.
%[RIGID ONLY?]

\begin{rmk}\label{rmk_simple} In Theorem \ref{iff}, the assumption that every quantifier-free type realized in $\B$ is equivalent in $\B$ to an $\LB$-formula follows from the assumption that there are only finitely many quantifier-free $n$-types realized in $\B$ for any $n \geq 1$.  To see why this is true, enumerate the quantifier-free $n$-types realized in $\B$, $\{q_s : s < m\}$, and note that there exists a quantifier-free formula $\theta_s \in q_t \setminus q_s$ for all $s \neq t$, $s<m$.  Then, for any $t<m$, $\M \vDash \forall x \left( q_t(x) \leftrightarrow \bigwedge_{s \neq t} \theta_s(x) \right)$.

The assumption that there are only finitely many quantifier-free $n$-types realized in $\B$ for any $n \geq 1$ holds in any countable $\omega$-homogeneous structure $\B$ in a finite signature that is uniformly locally finite (meaning for each $n$, there is a finite bound on the size of substructures of $\B$ generated by $n$ elements).  A proof of this result can be read in Corollary 7.4.2 of \cite{ho93}.
\end{rmk}

\begin{thm}\label{iff}  Fix locally finite ordered structures $\A$ and $\B$ and suppose that
$\A$ is a quantifier-free reduct of $\B$ (thus, for some $\kappa$, $|\A|=|\B|=\kappa$.)  Assume that $\B$ is saturated (i.e., $\kappa$-saturated.)
%Supose that $\A$ has finitely many quantifier-free $n$-types, for each $n \geq 1$, and  
%
Suppose that every quantifier-free type realized in $\B$ is equivalent in $\B$ to an $\LB$-formula. 
%and that there are only finitely many quantifier-free $n$-types for any $n \geq 1$.
%
Suppose that $\B$ has RP.  Then, $\A$ is a semi-retract of $\B$ if and only if $\A$ has RP.
\end{thm}

\begin{proof} The $\Rightarrow$ is from Corollary \ref{transfer}.  It remains to show $\Leftarrow$.  

Suppose that $\A$ has RP.  
%%%
Let $X = (a_i \mid i \in \A)$ enumerate $\B$ where $a_i = i$.  Since $\A$ has RP, by Theorem \ref{dictionary}, $\A$-indexed indiscernible sets have the modeling property.
%(see Definitions \ref{genind}, \ref{locbased}, \ref{mpd} and Proposition \ref{twodefs}).
%
Let $Y = (b_i \mid i \in \A)$ be an $\A$-indexed indiscernible set locally based on $X$.  By the saturation assumption, we know that we can witness $Y$ in $\B$, and not just in some elementary extension of $\B$ (see Remark \ref{shelah}).
Define $f: \A \raw \B$ to take $i \mapsto b_i$.
This is an injective map, since $a_i \neq a_j$ for all $i \neq j$ and $Y$ is $\A$-indexed indiscernible locally based on $X$.  It remains to show that
$$\A \xrightarrow{f} \B \xrightarrow{id} \A$$
\noindent is a semi-retraction.  

By Observation \ref{obs_reduct}, we already know that the identity map is qftp-respecting.
We must show the remaining properties from Definition \ref{def_semir}: (i) $\ii_1 \sim_\A \ii_2 \Rightarrow f(\ii_1) \sim_\B f(\ii_2)$ and  that (ii) $\jj \sim_\A id(f(\jj)) = f(\jj)$, for all tuples $\jj, \ii_1, \ii_2 \in {|\A|}^n$, for all $n<\omega$.  We have (i) as a direct consequence of $\A$-indexed indiscernibility.  We have (ii) from local basedness: for every finite subset of $\LB$-formulas $\Delta$, for any $\jj \in {|\A|}^n$, there is $\ii \sim_\A \jj$ from $\A$ such that $b_{\jj} \equiv^\B_{\Delta} a_{\ii}$.  In other words, $f(\jj) = b_{\jj} \equiv^\B_{\Delta} a_{\ii} = \ii$.
Given an arbitrary $\jj$, let $\Delta$ contain the formula equivalent to $\qftp^\B(f(\jj))$, and fix a corresponding $\ii \sim_\A \jj$.
Then, $f(\jj) \sim_\B \ii$, which implies that $f(\jj) \sim_\A \ii$, since $\A$ is a quantifier-free reduct of $\B$.  Since $f(\jj) \sim_\A \ii \sim_\A \jj$, we have that $f(\jj) \sim_\A \jj$, showing (ii).
\end{proof}

\begin{ex}\label{ordRadograph} Let $\B:=\R^<$ be the random ordered graph (the \fr~limit of finite ordered graphs) and $\A:= \B \uphp \{<\}$, so $\A$ is isomorphic to the rational linear order. It is easy to see that $\A$ is a quantifier-free reduct of $\B$, $\B$ is countably saturated, and both are locally finite and ordered.
%I am pretty sure
%$\Oo$ is in a finite relational signature (thus, finitely many quantifier-free $n$-types, for each $n$) and both have RP.
By Theorem \ref{iff}, $\A$ is a semi-retract of $\B$.

In order to find a semi-retraction $(g,f)$ between $\A$ and $\B$, one strategy is to take an indiscernible sequence $(g(i) \mid i \in \A)$ in $\R^<$, i.e. such that for all integers $n \geq 0$, for all $i_0<\ldots<i_{n-1}$ and $j_0<\ldots<j_{n-1}$ from $\A$, the map $g(i_k) \mapsto g(j_k)$ for all $k<n$ provides an isomorphism of ordered graphs in $\B$.
%$$(f(i_0),\ldots,f(i_{n-1})) \cong (f(j_0),\ldots,f(j_{n-1})) .$$
Such an indiscernible sequence is guaranteed to exist in $\B$, but in such a straightforward case, we can discern the existence of such a one directly: let $g(\A)$ be a copy of the countably infinite complete graph $K_\omega$ whose vertices form a dense linear order without endpoints under $<$ (guaranteed to exist in $\B$ by its saturation), where $g: \A \rightarrow \B$ is set up to preserve order.  Then $f$ can be taken to be merely the identity map, as follows
$$\A \xrightarrow{g} \B \xrightarrow{f=\id} \A .$$
This pair $(g,f)$ now witnesses that $\A$ is a semi-retract of $\R$.

Though both $\A,\B$ are well-known to have RP, this fact can be thought of as an instance of transfer by Corollary \ref{cor_fmla_free}.
\end{ex}

A similar argument can be made for the case $\A = (\Q,<)$ and $\B = (|\B|, E, \prec)$ where $\B$ is either of the ordered equivalence relation structures described in Example \ref{ex_RP} or Example \ref{ex_notRP}.

Next we look at three locally finite ordered structures known to have RP: the \textit{Shelah tree} $\I_\tr$, the \textit{strong tree} $\I_\str$ and the \textit{convexly ordered equivalence relation} $\I_\eq$.  The structure $\I_\eq$ has RP (Example \ref{ex_RP}) which fact is related to the usefulness of mutually indiscernible sequences in model theory (see Theorem III 7.12 (iii) in \cite{sh78} for an early example of mutually indiscernible sequences) even if RP is not explicitly mentioned.  An infinitary proof that shows that $\I_\tr$ has RP is given in the Appendix of \cite{sh78} and a finitary proof is given in Ch.2 $\S$2.2 Lemma 2 of \cite{nvt10}.  The fact that $\I_\str$ has RP is implicitly used in the proof of Thm III.7.11 in \cite{sh78} and used explicitly in the survey paper \cite{kks14}.  See \cite{sc15} for a more detailed discussion of this history.  

\begin{dfn}\label{trees}
\begin{itemize}
\item Define $\I_\tr$ to be the structure on $\Wo$ (finite sequences from $\omega$) in the signature
 $\{\unlhd, \wedge, \lx, \{P_n\}_{n \in \omega} \}$ where for all $\eta, \nu \in \Wo$, $\eta \unlhd \nu$ if and only if $\eta$ is an initial segment of $\nu$, $\wedge$ is the meet in the partial order $\unlhd$, $\lx$ is the lexicographic order on finite sequences, i.e. $\eta \lx \nu$ if and only if
$$\eta \unlhd \nu \textrm{~or~} \eta(  |\eta \wedge \nu| ) < \nu( |\eta \wedge \nu| ) ,$$

\noindent and $ \eta \in P_n $ if and only if $ |\eta| = n$, for all $n \in \omega$.
\item Define $\I_\str$ to be the structure on $\Wo$ in the signature $\{\unlhd, \wedge, \lx, \len \}$ where $\unlhd, \wedge, \lx$ are interpreted as in $\I_\tr$ and $\len$ is the preorder on $\mu, \nu \in \Wo$ defined by the lengths of the sequences:
$$\mu \len \nu \Leftrightarrow|\mu| < |\nu|$$
%i gives the column, j gives the row
\item Define $\I_\eq$ to be the structure on $\omega \times \omega$ in the signature $\{ E, \prec \}$ where for all $(i,j),(s,t) \in \omega \times \omega$, $(i,j)E(s,t) \Leftrightarrow i=s$ and 
$(i,j) \prec (s,t) \Leftrightarrow i<s \vee (i=s \wedge j<t)$.
\end{itemize}
\end{dfn}

We recall the following example of a semi-retraction that transfers RP from $\age(\I_\str)$ to $\age(\I_\eq)$.

\begin{prop}[\cite{sc21}]\label{treeprop} Let $\A$ be the structure on the underlying set $\omega \times \mathbb{Q}$ such that $\age(\A) = \age(\I_\eq)$ and each equivalence class in $\A$ is densely ordered by $\prec$.
Let $\B$ be the structure on the underlying set ${\mathbb{Q}}^{<\omega}$ such that $\age(\B) = \age(\I_\str)$ and the $\unlhd$-successors of any fixed node in $\B$ are densely ordered by $\lx$.  Then $\A$ is a semi-retract of $\B$.
\end{prop}

\begin{proof}   
Given $i \in \omega$, by the $i$th level in $\B$, we mean all sequences in ${\mathbb{Q}}^{<\omega}$ of length $i$, and by the $i$th equivalence class in $\A$, we mean $\{(i,x) \mid x \in \mathbb{Q} \}$.

Let $\eta_i = {\underbrace{( 0, \ldots, 0 )}_{2i}}$.  Let $g$ take the $i$th equivalence class in $\A$ into $\{\eta_i^\smf ( j ) \mid j \in \mathbb{Q}_{>0}\}$ in a way that preserves the order.
Let $f: \B \raw \A$ be the map that takes the $i$th level in $\B$ into the $i$th equivalence class in $\A$ in a way that preserves the order.
Then $\A$ is a semi-retract of $\B$ via $(g,f)$.
\end{proof}

The following example is an application of Theorem \ref{iff}.

\begin{ex}\label{treereduct}  Since $\I_\tr$ is ordered and has RP, $\age(\I_\tr)$ is an amalgamation class. Let $\B$ be the \fr~limit of $\I_\tr$ and let $\A$ be the reduct of $\B$ to the signature of $\I_\str$.  Thus, $\age(\A) = \age(\I_\str)$ and so $\A$ has RP and $\B$ has RP.  Since $\B$ is $\omega$-homogeneous and uniformly locally finite in a finite signature, $\Th(\B)$ is $\aleph_0$-categorical and has quantifier elimination, as in Remark \ref{rmk_simple}, and so $\B$ is $\aleph_0$-saturated and every quantifier-free type in $\B$ is equivalent to an $\LB$-formula in $\B$.  Since the conditions in Theorem \ref{iff} are satisfied and $\A$ has RP, it must be that $\A$ is a semi-retract of $\B$.
\end{ex}

%moved

\section{(Counter)example}\label{counterexample}
This section is inspired by an example of Ma\v{s}ulovi\'c from \cite{mas18} where a pre-adjunction (see Section \ref{pread}) is constructed to transfer RP from the category of finite naturally ordered Boolean algebras to the category of finite linearly ordered graphs.  Both classes were previously known to have RP (see \cite{rone77} and \cite{abha78} for linearly ordered graphs, and \cite{graro71} for Boolean algebras), but the mechanism of transfer is quite interesting with a number of applications.

%Let $(\R,<)$ denote the countably infinite random ordered graph (the Fra\"iss\'e limit of finite linearly ordered graphs) 
Let $\R$ denote the countably infinite Random graph (the Fra\"iss\'e limit of finite graphs) and let $\mathcal{H}_n$ denote the countably infinite random $n$-regular hypergraph (the Fra\"iss\'e limit of finite $n$-regular hypergraphs). Let $\Ba$ denote the countable atomless Boolean algebra (the Fra\"iss\'e limit of finite Boolean algebras) with its induced partial order denoted by $<^{\Ba}$. Below we show that $\R$ and $\mathcal{H}_n$ are semi-retracts of $\Ba$. By Theorem \ref{thm_transembd}, this gives us upper bounds on Ramsey degrees for embeddings of finite graphs, respectively finite $n$-regular hypergraphs, in terms of Ramsey degrees for embeddings of finite Boolean algebras, and consequently on Ramsey degrees of substructures by Corollary \ref{degrees_bound}. See Remark \ref{rmkba} for a more detailed analysis. 

\subsection{Semi-retracts of the countable atomless Boolean algebra}\label{4.01}

\begin{thm}\label{4.2} The countable random graph $\R$ is a semi-retract of the countable atomless Boolean algebra $\Ba$.
\end{thm}

\begin{proof} %Let $R$ on $\Ba$ be defined as in in the proof of Theorem \ref{4.1}.
We define a graph relation $R$ on $\Ba$ as Ma\v{s}ulovi\'c does on power set algebras of finite sets in \cite{mas18}: $(a,b)\in R$ if and only if  $a \neq b$ and $a\wedge b\neq \mathbf{0}$.
We will build a graph embedding $g:\R\raw (\Ba,R)$ that will be qftp-respecting when $\Ba$ is considered as a Boolean algebra. Since $\R$ is universal for countable graphs, there must be a graph embedding $f: (\Ba,R)\raw \R$. Clearly, $f$ will remain qftp-respecting with respect to $\Ba$ in the signature of Boolean algebras. Since $g$ and $f$ are graph embeddings, $f\circ g$ is a graph embedding from $\R$ to itself. Thus $(g,f)$ will be a semi-retraction between $\R$ and $\Ba$. 

Let $V=\{v_n:n\in\omega\}$ be an enumeration of vertices in $\R$. Recall that an \textit{antichain} in a Boolean algebra is a collection of non-zero elements such that every two distinct elements have zero meet. 
Pick an infinite antichain $B=\{b_n:n\in\omega\}$ in $\Ba$. %and split it into two disjoint sets $A=\{a_n:n\in \omega\}$ and  
For any $b_n\in B,$ let $\{b_n^i:i\in\omega\}$ be an antichain $\leq^{\Ba}$-below $b_n.$ %Let $b':V\to A$ be a bijection a 
Define $g:\R\to \Ba$ by \[g(v_n)=b_n\vee\bigvee_{i<n}b_i^n(\varepsilon),\] where 
\[
b_i^n(\varepsilon)=\begin{cases}
b_i^n \textrm{\ if } (v_i,v_n)\in E\\
\mathbf{0} \textrm{\ otherwise}.
\end{cases}
\]
We will show that $g$ is qftp-respecting:  
%Let $f'$ be any graph embedding from $(\Ba,R)$ to $\R$. Clearly, $f'$ is qftp-respecting if we pass to the reduct $\Ba$. By the same argument as in Theorem \ref{4.1}, $g'$ will also be qftp-preserving. Since $g'$ and $f'$ are graph embeddings, $f'\circ g'$ is an embedding of $\R$ into itself.  
Every quantifier-free type in $\R$ is determined by $E$ on pairs of points, and that needs to be reflected in the image of $g$. 
By the definition of $g$, if $(v_i,v_n)\in E$ and $i<n$ then $g(v_i)\wedge g(v_n)=b_i^n \neq\mathbf{0}$. If $(v_i,v_n)\notin E,$ then %$g(v_i)\leq^{\Ba} e(v_i)\vee b_{q_i}\vee \bigvee\{b(j,i):j<i\}$ and $g(v_n)\leq^{\Ba} e(v_n) \vee b_{q_n}\vee\bigvee\{b(j,n):j<n, j\neq i\}$. 
%%need to know, whatever it is, we won't hit the same pair $(q_i,q_j)$ twice.
%Thus $g(v_i)$ and $g(v_n)$ are below joins of disjoint finite antichains and so 
$g(v_i)\wedge g(v_n)=\mathbf{0}$. %Also, $g$ is order preserving, since $e$ is. 
 
%not quite graph embedding.
 %Since $f$ is a graph embedding, $f$ is qftp-respecting when we pass to the reduct $(\Ba,\prec)$. We now verify that so is $g$.
% Any quantifier free type $p$ in $(\R,<)$ is given by $E$ and $<$ on pairs of vertices. 
Any Boolean expression with variables $x_0,\ldots, x_{n-1}$  has its unique equivalent full disjunctive normal form, that is, a disjunction of clauses, where each clause consist of conjuction of literals $x_i$ or $\neg x_i$ so that each of the variables $x_i$ appears in every clause. We will show that an $n$-type of elements of $\Ba$ in the image of $g$ depends only on the type of the  graph they encode:
%I added "ordered"
 \begin{enumerate} 
\item\label{ba1} For every triple of distinct numbers $i,j,k$: $g(v_i)\wedge g(v_j)\wedge g(v_k)=\textbf{0}$.
 \item Whenever $i< j,$ we have that $g(v_i)\wedge g(v_j)=\begin{cases}
 \mathbf{0} \text{\ \ if } (v_i,v_j)\notin E\\
 b_i^j \text{\ \ if } (v_i,v_j)\in E.
 \end{cases}$
 \end{enumerate}
 It follows that every non-zero clause has at most two positive literals. Further,
 \begin{enumerate}
 \item[(3)] For any $i\neq j$, \[g(v_i)\wedge \neg g(v_j)=g(v_i)\wedge \neg(g(v_i)\wedge g(v_j))=
 \begin{cases}
 g(v_i)  \text{\ if } (v_i,v_j)\notin E\\
 g(v_i) \wedge \neg b_i^j \text{\ if } (v_i,v_j)\in E \text{\ and \ } i<j \\
 g(v_i) \wedge \neg b_j^i  \text{\ if } (v_i,v_j)\in E \text{\ and } i>j.
 \end{cases}
 \]
 \item[(4)] For $i\neq j,$ $\neg g(v_i)\wedge \neg g(v_j)=\neg (g(v_i)\vee g(v_j))$ by de Morgan's law. 
\end{enumerate}
Let $v_{i_0}, v_{i_2},\ldots, v_{i_{n-1}}$ be $n$ distinct vertices of $\R$. We will consider Boolean expression in $g(v_{i_0}),\ldots, g(v_{i_{n-1}})$ in full disjunctive normal form. From the analysis above, we have three cases.
\begin{enumerate}
	\item[(i)] If a clause contains no positive literal, then it is equal to $\neg (\bigvee_{k=0}^{n-1}g(v_{i_k}))$, which is a positive element disjoint from any $g(v_{i_k}).$ 
 	\item[(ii)] If a clause has one positive literal $g(v_{i_k})$ and all other $n-1$ negative, it equals to 
 	%\begin{eqnarray*}
 	%g(v_{i_k})\wedge \neg(\bigvee \{b(i_j,i_k): (v_{i_k},v_{i_j})\in E, i_j \neq i_k \})\\
 		%=
   \[g(v_{i_k})\wedge \neg(\bigvee \{g(v_{i_k})\wedge g(v_{i_j}):(v_{i_k},v_{i_j})\in E\}),\]
 	%\end{eqnarray*}
  which is $\leq^{\Ba} g(v_{i_k})$ and disjoint from every $g(v_{i_k})\wedge g(v_{i_j})$ for $j\neq k.$
 	\item[(iii)] If a clause contains two positive literals $g(v_{i_k}), g(v_{i_l})$,  then it is equal to $g(v_{i_k})\wedge  g(v_{i_l}).$
 
 \end{enumerate}
Since $ \{b_i^j: i<j<\omega\}$ is an antichain, we get that any two distinct clauses have empty meet.  Therefore, it is not possible to obtain one clause below a join of the others, unless it equals to zero. It means that there are no non-trivial equations with quantifier free formulas about an $n$-tuple in the image of $g$ other than comparison with $\mathbf{0}$. Since clauses in items (i) and (ii) are never equal to $\mathbf{0},$ the entire type is decided by the clauses as in (iii), which corresponds exactly to the type of the finite graph in the preimage.
\end{proof}

\begin{rmk}\label{rmkba} We say that $A$ is a Ramsey object in $\K$ if for all $B \in \K$, $(A,B)$ is a Ramsey duo for $\K$.
The class of finite Boolean algebras has RP (as in Example \ref{ex_RP}) while only complete or empty graphs are Ramsey objects in the class of finite graphs, as shown in \cite{rone75}. Therefore,  to secure transfer of RP by semi-retractions, the requirement %that $\age(\Ba)$ be rigid is essential in
on rigidity of $\B$ in Corollary \ref{cor_fmla_free} or on relational signature in Theorem \ref{fmla_free_relational} cannot be completely removed. 
The reason in the example of Theorem \ref{4.2} is that there can be two distinct copies $\Gamma_1,\Gamma_2$ of the same finite graph in $\R$ such that $f^{-1}(\Gamma_1)$ and $f^{-1}(\Gamma_2)$ have the same quantifier-free type and generate the same  Boolean subalgebra $A$ of  $\Ba$ (thanks to function symbols in the signature of Boolean algebras). This would not have been possible if $\age(\Ba)$ were rigid, since a partial isomorphism from $f^{-1}(\Gamma_1)$ to $f^{-1}(\Gamma_2)$ given by having the same quantifier-free type would extend to a non-trivial automorphism of $A$.

\begin{comment}
This is not in contradiction with Theorems \ref{masresult},\ref{preadjage}, and \ref{thm_transembd}, and Corollary \ref{cor_fmla_free}, because
%the fact that semi-retractions define pre-adjunctions (see Section \ref{category}), which transfer RP, 
%is that the respective pre-adjunction is considered with all embeddings while 
the transfer principle of the RP in Corollary \ref{cor_fmla_free} is with respect to substructures 
whereas the transfer principle of the RP in Theorems \ref{masresult} and \ref{thm_transembd} is with respect to embeddings.
\end{comment}

 Theorem \ref{thm_transembd} applied to $(g,f)$ shows that the finite Ramsey degree of  a finite graph $\Gamma$ in the class of finite graphs with embeddings is bounded above by the Ramsey degree of the Boolean subalgebra of $\Ba$ generated by $g(\Gamma)$ in the class of finite Boolean algebras with embeddings (which is equal to $n!$, where $n$ is the number of atoms in $\left<g(\Gamma)\right>$).  It can be derived from  results in \cite{abha78} and \cite{rone77} that the exact Ramsey degree for embeddings of a graph on $m$ vertices is $m!$. Thus our estimate is optimal only in the case of discrete graphs.
\end{rmk}

\begin{qu}
The example of a semi-retraction on Theorem \ref{4.2} was inspired by an example in \cite{mas18}. However, Ma\v{s}ulovi\'c's example is about the classes of finite ordered graphs and finite naturally ordered Boolean algebras, rather than their unordered versions presented here. Is the random ordered graph (the Fra\"iss\'e limit of finite linearly ordered graphs) a semi-retract of the countable atomless Boolean algebra with a generic normal order (the Fra\"iss\'e limit of finite naturally ordered Boolean algebras)?  
\end{qu}

Let $\mathcal{H}_n=(V,E)$ be the countable random $n$-regular hypergraph. Ma\v{s}ulovi\'c commented in his paper \cite{mas18} that the method he used to witness the Ramsey property of linearly ordered graphs by that of finite naturally ordered Boolean algebras did not generalize to higher arity hypergraphs. We were able to overcome this obstacle in the unordered case. As in graphs, one needs to ensure that the $g$-part of the semi-retraction maps to a portion of the Boolean algebra, where every type is determined only by the hypergraph relation. This can be achieved by a construction similar to that in Theorem \ref{4.2}, while making sure that all $k$-tuples for $k<n$ have the same type. %We present only the unordered version here. 

\begin{thm}\label{4.3} For any integer $n\geq 2$, $\mathcal{H}_n$ is a semi-retract of $\Ba$.
\end{thm}
\begin{proof}
%\subsection{Embedding $n$-regular hypergraphs into countable atomless Booelan algebra}\label{4.3}
We define an $n$-ary hypergraph relation $H_n$ on $\Ba$ by $(b_0,\ldots,b_{n-1})\in H_n$ iff 

\noindent $\bigwedge_{i<j<n} b_i \neq b_j$ and $\bigwedge_{l=0}^{n-1} b_l\neq\textbf{0}$. %$\bigwedge_s\in S b_s$ for every $S\subset [n].$ 
We will define an embedding $g:\mathcal{H}_n\to (\Ba, H_n)$ and by universality of $\mathcal{H}_n,$ there is an embedding $f:(\Ba,H_n)\to\mathcal{H}_n$.   Since $H_n$ is quantifier-free definable in $\B$, taking  the reduct $\Ba$ of $(\Ba, H_n),$ $(g,f)$ will be a semi-retraction between $\Ba$ and $\mathcal{H}_n$, as in the case of the random graph. However, we need to be more careful when defining $g$ than in the graph case -- in $(\Ba,H_n)$ there are different  $<n$-types, which is not true in $\mathcal{H}_n$

%The simplest idea to define the hypergraph relation on $\B$ by $(a_0,\ldots,a_{n-1})\in H_n$ iff $\bigwedge_{l=0}^{n-1}\neq\textbf{0}$ does not work, since all $<n$-tuples in $\mathcal{H}_n$ have the same type, but that does 

For every $1\leq k\leq n,$ let $B_k=\{b_{\ov{\imath}}: \ov{\imath}\in\omega^{[k]}\}$ be an antichain in $\Ba$, where $\omega^{[k]}$ denotes all strictly increasing sequences in $\omega$ of length $k$. We further require that if $\ov{\imath}$ is an initial segment of $\ov{\jmath},$ then $b_{\ov{\jmath}} <^{\Ba} b_{\ov{\imath}}.$ Let $\{v_l:l\in \omega\}$ be an enumeration of vertices of $\mathcal{H}_n$ and let $E$ denote the set of hyperedges. We define $g:\mathcal{H}_n\to (\Ba,H_n)$ by 
\[
g(v_l)=\bigvee_{\ov{\imath}(|\ov{\imath}|-1)=l, |\ov{\imath}|<n} b_{\ov{\imath}}\vee \bigvee_{\ov{\imath}(|\ov{\imath}|-1)=l,|\ov{\imath}|=n, (v_{i(0)},\ldots,v_{i(n-1)})\in E} b_{\ov{\imath}}.
\]
Since each $B_k$ is an antichain, we have that for any $\ov{\imath}\in \omega^{[k]}$
\[
\bigwedge_{l=0}^{k-1} g(v_{i(l)})\neq\mathbf{0} \text{\ iff } k<n, \text{\ or } k=n \text{ \ and } (v_{i(0)},\ldots,v_{i(n-1)})\in H_n.
\]

We have that $g$ and $f$ are hypergraph embeddings and thus $fg$ is an embedding. As in the proof of Theorem \ref{4.2}, one can show that every $N$-type for $N\geq n$ in $g(\mathcal{H}_n)$ is determined by $n$-types. Here it is crucial that every two $k$-tuples for $k<n$ have the same type and therefore $n$-types are determined solely by the hypergraph relation.
We can conlude that $g$ considered as a map $\mathcal{H}_n\to \Ba$ is qftp-respecting. 
\end{proof}

\begin{rmk}
As in \ref{rmkba},
by Theorem \ref{thm_transembd}, the Ramsey degree for embeddings of a finite hypergraph $H$ in the class of finite hypergraphs is bounded above by the Ramsey degree of $\left<g(H)\right>$ in the class of finite Boolean algebras with embeddings. %{\color{red} compare to what is known in the literature, before offering a bound on the Ramsey degree for embeddings, e.g. $n!$.}
\end{rmk}

%VERIFY that this is a hypergraph embedding and that it remains a quantifier free type respecting.

%%%%%%%%%%%%%%%%%%%%%%%%%%%%%%%%%%%%%%%%%%%%%%%%%%%%%%%%%%%%%%%%%%%%%%
\section{Formula-free approach to transferring the RP}\label{fmla}
%%%%%%%%%%%%%%%%%%%%%%%%%%%%%%%%%%%%%%%%%%%%%%%%%%%%%%%%%%%%%NEW START
%In Section \ref{prelim}, we introduced the Ramsey property and finite Ramsey degrees for substructures and embeddings. 
In prior work \cite{sc21}, Corollary \ref{transfer} demonstrated how semi-retractions transfer RP when $\A$ and $\B$ are both ordered and locally finite.
In this section we show in Corollary \ref{cor_fmla_free} that some of the assumptions in Corollary \ref{transfer} may be dropped. Namely, we only need rigidity of the structures in $\age(\B)$ and we only need local finiteness of $\A$.  In fact, the transfer can be done locally for a pair of finite structures $(A,B)$ in $\A$, as demonstrated in Corollary \ref{cor_fmla_free_duo}. Moreover, if the signatures of $\A$ and $\B$ are relational, we no longer need rigidity in $\age(\B)$ as shown in Theorem  \ref{fmla_free_relational}.

In Remark \ref{rmkba}, we noted that the rigidity assumption in Corollary \ref{cor_fmla_free_duo} or  the relational language in Theorem \ref{fmla_free_relational} cannot be removed.  However, if we consider the Ramsey property for embeddings rather than for substructures, which is the natural framework in dynamical applications, we obtain a transfer principle for finite Ramsey degrees for embeddings (Theorem \ref{thm_transembd}) that does not assume more about $\A$ and $\B$ other than that $\A$ is a locally finite semi-retract of $\B$. 
We include examples that demonstrate the necessity of the assumptions in Corollary \ref{cor_fmla_free} and Corollary \ref{cor_fmla_free_duo}.
%In Theorem \ref{thm_transembd}, we show that finite Ramsey degrees for embeddings transfer to locally finite semi-retracts. 
%and we give an example to demonstrate that local finiteness cannot be removed.  We then discuss how the other Ramsey notions behave under semi-retractions. 

\begin{thm}\label{thm_transembd}
	Suppose that $\A$ is a locally finite semi-retract of $\B$ via $(g,f)$. Let $A \in\age(\A)$. Suppose that $\left<g(A)\right>_{\B}$ has Ramsey degree $d$ for embeddings in $\age(\B)$. Then $A$ has Ramsey degree $\leq d$ for embeddings in $\age(\A)$.
\end{thm}

\begin{proof} 
Let $A_0, B_0$ be finite substructures of $\A$ and let $A=fg(A_0)\cong A_0$ and $B=fg(B_0)\cong B_0$. Let $c:\emb(A,\A)\to\{0,1,\ldots, r-1\}$  be a coloring.  We will make use of Lemma \ref{compactness_emb_degrees} in proving this result.
	
Let $A'=\left<g(A_0)\right>_{\B}$ and $B'=\left<g(B_0)\right>_{\B}$. Given any $e \in \emb(A',B')$, the domain of $f(e \uphp g(A_0))$ is $fg(A_0)=A$. For any $\ov{x}\in g(A_0),$ we have % and for any pairs $\la x_i,y_i \ra \in e \uphp g(A_0)$, for $i<n$ for some integer $n$, 
$\ov{x} \sim_\B e(\ov{x})$ since $e$ is an embedding, and thus by the qftp-respecting property of $f$, $f(\ov{x}) \sim_\A f(e(\ov{x}))$.  By these two observations, $f(e\uphp g(A_0)) \in \emb(A,\A)$.
Thus, we may define an induced coloring $c_0: \emb(A',\B) \raw \{0,1,\ldots, r-1\}$ by $c_0(e)=c(f(e \uphp g(A_0)))$.

By assumption, there exists  $h \in \emb(B',\B)$ such that $c_0$ restricted to $h\circ\emb(A',B')$ takes on at most $d$ colors. We let $k=f(h\restriction g(B_0))$. By an argument similar to the one in the previous paragraph, we have that $k\in\emb(B,\A)$.  We claim that $c$ restricted to $k \circ \emb(A,B)$ takes on at most $d$ colors:

%To prove the latter claim, 
Fix any $j \in \emb(A,B)$.  It is enough to show that $c(k \circ j) = c_0(h \circ j')$ for some $j' \in \emb(A',B')$.  By the definition of $A$ and $B$, %for any $j \in A \times B$, $(fg)^{-1}(j) \in A_0 \times B_0$.  Moreover, 
and since $fg: \A \to \A$ is an embedding, we have that $(fg)^{-1}(j) \in \emb(A_0,B_0)$.
%
%Any embedding $e:A'\raw \B$ is uniquely determined by the restriction of $e$ to the generators of $A'$, namely $e\restriction g(A_0)$.  
%By the definition of $A'$, a partial embedding defined on the generators $g(A_0)$ of $A'$ extends uniquely to an embedding in $\emb(A',\B)$.
%
%Thus, we can set 
Let $j'\in \emb(A',B')$ to be unique embedding extending the partial embedding $g((fg)^{-1}(j))=f^{-1}(j):g(A_0) \to g(B_0)$.  
%yes! that restriction on the codomain was needed.
Then $j'$  satisfies $c(k \circ j) = c_0(h \circ j')$.  To see this, note that $c_0(h \circ j')$ is defined to be $c(f((h \circ j') \uphp g(A_0)))$.  However, $(h \circ j') \uphp g(A_0) = h \circ (j' \uphp g(A_0)) = h \circ (f^{-1}(j) \uphp g(A_0)) = (h \uphp g(B_0)) \circ (f^{-1}(j) \uphp g(A_0)) = (f^{-1}(k)) \circ (f^{-1}(j) \uphp g(A_0))$.  Now applying $f$ to $(h \circ j') \uphp g(A_0)$ we obtain $f$ applied to $(f^{-1}(k)) \circ (f^{-1}(j) \uphp g(A_0))$ which is $k \circ j$.  Thus $c(f((h \circ j') \uphp g(A_0))) = c(k \circ j)$, as desired.
\end{proof}

In Proposition \ref{TwoDegrees}, we recalled
%reminded 
the exact formula relating the Ramsey degree for embeddings and the Ramsey degree for substructures in case of a locally finite class. Therefore Theorem \ref{thm_transembd} immediately yields the following corollary.

\begin{cor}\label{degrees_bound}
Suppose that $\A$ is a locally finite semi-retract of $\B$ via $(g,f)$. Let $A \in\age(\A)$. Suppose that $\left<g(A)\right>_{\B}$ has Ramsey degree $d$ for embeddings in $\age(\B)$. Then $A$ has Ramsey degree $\leq d/|\Aut(\A)|$ for substructures in $\age(\A)$.
\end{cor}

In Theorem \ref{4.2}, we provided an example of the random graph as a semi-retract of the countable atomless Boolean algebra. While the class of finite Boolean algebras is a Ramsey class, the class of finite graphs is not, so we cannot simply say that semi-retractions transfer the Ramsey property for structures. However, the existence of the semi-retraction provides a bound on finite Ramsey degrees of finite graphs as per Corollary \ref{degrees_bound}. We point out that the reason that Section \ref{counterexample} does not provide a direct transfer of the Ramsey property is that the signature of Boolean algebras contains function symbols, $\age(\Ba)$ is not rigid, and the semi-retraction defined in the proof of Theorem \ref{4.2} allows two distinct finite isomorphic graphs to generate the same Boolean algebra via preimages under $f$. We show in  Theorem \ref{fmla_free_relational}  that restricting to relation symbols and in Corollary \ref{cor_fmla_free} that restricting to $\B$ with age consisting to rigid structure, suffice to transfer the Ramsey property, respectively.

%%%this fact was moved to Preliminaries.

%MOTIVATION WHY CONSIDER SPECIAL CASES WITH RIGIDITY

Example \ref{noorder} explained why rigidity is essential for applications of Ramsey theory to indiscernible sequences. Thus we present a few special cases of Theorem \ref{thm_transembd} in the rigid setting.

\begin{cor}\label{cor_fmla_free_degrees} Let $\A, \B$ be structures and let $\K:=\age(\A), \K' = \age(\B)$.  Assume $\A$ is {locally finite} and $\K'$ consists of rigid elements.  Suppose that $\A$ is a semi-retract of $\B$ by the maps $(g,f)$.   For any finite substructure $A \subseteq \A$, letting $A':=\la g(A) \ra_\B$, if $A'$ has finite Ramsey degree $d$ in $\K'$, then $A$ has finite Ramsey degree in $\K$, and in fact, $d(A,\K) \leq d$.
\end{cor}

\begin{proof}  By Proposition \ref{TwoDegrees},  $d(A,\K) \leq d_e(A,\K)$.  If $\K'$ consists of rigid elements, then $d_e(A',\K') = (1) \cdot d(A',\K')$.  Thus we have the inequalities:
$$d(A,\K) \leq d_e(A,\K) \leq d_e(A',\K') = d(A',\K')=d$$
where the middle inequality follows from Theorem \ref{thm_transembd}.
\end{proof}

%%%%%%%%%%%%%%%%%%%%%%%%%
%Ramsey duo transfer as a corollary
%%%%%%%%%%%%%%%%%%%%%%%%%%%%
Corollary \ref{cor_fmla_free_duo} follows easily from the argument for Theorem \ref{thm_transembd}: we merely need to be aware of how we fix $B$ in the beginning and set $d=1$.  An alternative argument, not as a corollary of Theorem \ref{thm_transembd}, is presented in the Appendix in Section \ref{ModelTheoryFunctionalSection}.

\begin{cor}\label{cor_fmla_free_duo} Let $\A, \B$ be structures  and suppose $(g,f)$ is a semi-retraction between $\A$ and $\B$.   
Suppose that $\ov{a}, \ov{b}$ are finite tuples from $\A$ that generate finite substructures $A, B$, respectively, of $\A$, such that $g(\ov{a}), g(\ov{b})$ generate rigid substructures $A_0, B_0$, respectively, of $\B$.  If $(A_0,B_0)$ is a Ramsey duo for $\B$, then $(A,B)$ is a Ramsey duo for $\A$.  
%$\age(\B)$ consists of rigid elements and $\A$ is locally finite.  If $\B$ has RP, then $\A$ has RP.
\end{cor}

Corollary \ref{cor_fmla_free} is a special case of Corollary \ref{cor_fmla_free_duo}.
\begin{cor}\label{cor_fmla_free} Let $\A, \B$ be structures  and suppose $(g,f)$ is a semi-retraction between $\A$ and $\B$.   
Suppose that $\A$ is locally finite, and $\age(\B)$ consists of rigid structures.  If $\B$ has RP, then $\A$ has RP.
%$\age(\B)$ consists of rigid elements and $\A$ is locally finite.  If $\B$ has RP, then $\A$ has RP.
\end{cor}

To show the necessity of the assumptions in Corollary \ref{cor_fmla_free}, we give an example of a structure $\A$ that is a semi-retract of $\B$, where each structure is in a non-relational signature, such that $\A$ fails to be locally finite, the element in $\age(\B)$ is not rigid, and RP fails to transfer from $\B$ to $\A$.

\begin{ex}\label{notlocfin}    
Let $\A = (\Z,s)$, $\B = (\Z, s, p)$ where $s, p$ are interpreted as the successor and predecessor functions, respectively, that is, $s(z)=z+1$ and $p(z)=z-1$ for every $z\in\Z$.  Note that any finite subset of $\B$ generates $\B$, and so $\B$ trivially has RP.  On the other hand, any finite subset $X$ of $\A$ generates the tail $[a, \infty) \subseteq \Z$ where $a = \min X$.  

To see that $\A$ fails to have RP, let $A = B = [0,\infty)$.  For any copy of $A$, there is a unique element $a \in A$ ($\min A$) that generates all of $A$.  Color copies of $A$ in $\A$ red if that unique element is odd, blue, if it is even.  There is no copy of $B$ in $\A$ that is homogeneous for this coloring on copies of $A$.  

It remains to verify that the identity maps give a semi-retraction $\A \xrightarrow{\id} \B \xrightarrow{\id} \A$, but this is the case because $\A$ and $\B$ are quantifier-free interdefinable (see Remark \ref{interdef}).  In the case of the predecessor function, we can define it from the successor function by switching variables: For any $a_1, a_2 \in |\A|=|\B|=\Z$, $\B \vDash p(a_2)=a_1 \Leftrightarrow \A \vDash s(a_1)=a_2$.
\end{ex}

To show the necessity of the assumptions in Corollary \ref{cor_fmla_free_duo}, we give an example of a structure $\A$ that is a semi-retract of $\B$, where $\A$ is in a non-relational signature, such that all structures in $\age(\B)$ are finite and rigid, and a Ramsey duo for $\B$ fails to transfer to a Ramsey duo for $\A$, specifically because the pair in $\A$ is a counterexample to local finiteness in $\A$.

\begin{ex}\label{succ} Let $\A = (\Z, s)$ be the structure where $s$ is interpreted as the successor function on $\Z$.  
Let $\B = (\Z, \{R_{(n,\ov{k})}\})$ 
where $R_{(n,\ov{k})}(a_0,\ldots,a_{n-1})$ for $\ov{k}=(k_0,\ldots,k_{n-1})$ holds exactly of increasing $n$-tuples $a_0<a_1<\ldots < a_{n-1}$ from $\Z$ such that $a_{i+1}-a_i=k_i$, for all $i<n$.
The identity maps $(g,f):=(\textrm{id},\textrm{id})$ give a semi-retraction because $\A$ and $\B$ are quantifier-free interdefinable (see Remark \ref{interdef}).

%We have already argued in Example \ref{notlocfin} that $A$ is not a Ramsey object in $\age(\A)$, and in this example we will look at how this impacts the mechanics of our argument for Theorem \ref{fmla_free}.  

Using the notation of Corollary \ref{cor_fmla_free_duo}, let $a:=0, b:=0$ in $\A$, and $a_0:=g(0)=0, b_0:=g(0)=0$ in $\B$.  Define substructures of $\A$: $A:=\la a \ra_\A = [0,\infty)$, $B:=\la b \ra_\A = [0,\infty)$.  Define substructures of $\B$: $A_0:=\la g(a_0) \ra_\B = \{0\}$, $B_0:=\la g(b_0) \ra_\B = \{0\}$.  Clearly $(A_0,B_0)$ is a Ramsey duo for $\B$.  Moreover, all structures in $\age(\B)$ are finite, and thus rigid.

However, the structures generated by $a$ and $b$ in $\A$ are not finite, so the assumptions of Corollary \ref{cor_fmla_free_duo} are not satisfied.  And indeed, the conclusion of Corollary \ref{cor_fmla_free_duo} is not achieved: $(A,B)$ is not a Ramsey duo for $\A$, as we saw in Example \ref{notlocfin}.
\end{ex}

Example \ref{succ} can be modified slightly to satisfy the assumptions of Corollary \ref{cor_fmla_free}.  %In Example \ref{pred}, both $\A$ and $\B$ trivially have RP because for all $A, B$ in the age of one of the structures, there is at most one copy of $A$ in $B$.

\begin{ex}\label{pred} Let $\A = (\Nn, p)$ be the structure where $p$ is interpreted as the predecessor function on the positive integers, i.e. $p(n+1)=n$, for all $n \in \Nn$, and $p(0):=0$.  Let $\B = (\Nn, p, s)$ where $p$ is defined as in $\A$ and $s$ is the successor function on $\Nn$.
%Let $\B = (\Nn, \{R_{(n,\ov{k})}\},0)$, where the predicates $R_{(n,\ov{k})}$ are defined as in Example \ref{succ}.
%
The identity maps $(g,f):=(\textrm{id},\textrm{id})$ give a semi-retraction because $\A$ and $\B$ are quantifier-free interdefinable (see Remark \ref{interdef}).  In the case of the successor function, we can define it from the predecessor function by switching variables: for any $a_1, a_2 \in |\A|=|\B|=\Nn$, $\B \vDash s(a_1)=a_2 \Leftrightarrow \A \vDash p(a_2)=a_1 \wedge a_1 \neq a_2$.

It is clear that $\A$ is locally finite.
The structure $\A$ trivially has RP because for all $A, B$ in $\age(\A)$, there is at most one copy of $A$ in $B$.  The structure $\B$ trivially has RP because $\age(\B)$ consists of one element (as in Example \ref{notlocfin}) however this one element (which is all of $\B$) happens to be rigid, because the element $0$ must be fixed by any automorphism of $\B$.  Thus, the assumptions (and conclusion) of Corollary \ref{cor_fmla_free} are satisfied in this example.
%We have already argued in Example \ref{notlocfin} that $A$ is not a Ramsey object in $\age(\A)$, and in this example we will look at how this impacts the mechanics of our argument for Theorem \ref{fmla_free}.  
\end{ex}

%%%%%%%%%%%%%%%%%%%%%%%%%%%%%%%%%%
%dana's passing text is here
%%%%%%%%%%%%%%%%%%%%%%%%%%%%%%%%%%

\begin{thm}\label{fmla_free_relational} Let $\A, \B$ be structures each in %(possibly distinct) 
relational signatures and suppose that $\A$ is a semi-retract of $\B$, and let $A\in\age(\A).$  If $\la g(A)\ra_{\B}$ has Ramsey degree $d$ in $\age(\B)$, then $A$ has Ramsey degree bounded by $d$ in $\A$.
\end{thm}

It would be possible to prove Theorem \ref{fmla_free_relational} with no mention of function symbols, but there is a property, the ``restricted inverse images under $f$'' property (defined in Definitions \ref{restricted_rlnl} and \ref{restricted}, that could be of independent interest in the functional case.  For this reason, we pursue a slightly longer development of the argument than needed here.  The reader who would like to pursue the additional details for a proof of Corollary \ref{cor_fmla_free} using the ``restricted inverse images under $f$'' property, is invited to read Section \ref{ModelTheoryFunctionalSection} within the Appendix.

%%%%%%%%%%%%%%%%%%%%%%%%%%%%%%%%%%%%%%%%%%%%%%%%
%relational
%%%%%%%%%%%%%%%%%%%%%%%%%%%%%%%%%%%%%%%%%%%%%%%%%%%
The argument for Theorem \ref{thm_transembd} is of a category-theoretic nature and so certain details at the level of substructures may not be immediately evident.  For example, consider Definition \ref{restricted_rlnl} and what it would take for a certain finite $B_0 \subseteq \B$ to have this property.

\begin{dfn}\label{restricted_rlnl} Fix structures $\A, \B$ in relational signatures, finite substructures $A, B \subseteq \A$ and an injection $f: \B \raw \A$.  Fix substructures $A_0, B_0$ from $\B$ such that $||A_0||=||A||$.
We say that $B_0$ has the \textit{relational restricted inverse images under $f$ property for $A$ witnessed by $A_0$} if for any $C_1 \subseteq f(B_0)$ such that $C_1 \cong_{\LA} A$,
$f^{-1}(C_1) \cong_{\LB} A_0$.
\end{dfn}

It turns out that if $\A$ is a semi-retract of $\B$ via $(g,f),$ so long as $B_0$ is isomorphic to an image under $g$ of a structure in $\A$ containing $A$ as a substructure, then $B_0$ has restricted inverse images under $f$ for $A.$ %the isomorphism types in $B_0$ are restricted.

\begin{prop}\label{prop_rlnl_restricted} Fix structures $\A$ and $\B$ in relational signatures.  
Fix  a semi-retraction $(g,f)$ between $\A$ and $\B$ and finite substructures $A, B \subseteq \A$.  Let $A_0 = g(A), B_0=g(B), A_1=f(A_0)$.  Then, for any $B_0' \cong_{\LB} B_0$, $B_0'$ has the relational restricted inverse images under $f$ property for $A$ witnessed by $A_0$.
\end{prop}

%%%%%%%%%%%%%
%short mention of functional case
%%%%%%%%%%%%%%%%%%%%%
To develop the proof, we first state the analogues of Definition \ref{restricted_rlnl} and Proposition \ref{prop_rlnl_restricted} in the case that the signatures contain function symbols.  Since Definition \ref{restricted} seems technical at first, we motivate it with Example \ref{trees_example}.

\begin{ex}\label{trees_example}  Let $\A$ and $\B$ be as defined in Proposition \ref{treeprop}.
Fix a quantifer-free type $p(x_0,x_1) = \{E(x_0,x_1), E(x_1,x_2), x_0\prec x_1\prec x_2\}$ in $\A$ and let $\ov{a}$ be a realization of this type in $\A$.  Let $\ov{b}$ be a finite tuple from $\A$ containing the elements of $\ov{a}$.  Let $(g,f)$ be the semi-retraction described in Proposition \ref{treeprop}.  Define $\ov{a}_0:=g(\ov{a})$, $\ov{b}_0:=g(\ov{b})$.

Consider the various quantifier-free types of $\{b_0 \lx b_1 \lx b_2\}$ in the tree $\B:=\I_\str$ that map to copies of $\ov{a}$ in $\A$ under the map $f: \B \raw \A$: 

\begin{enumerate}
\item $\neg ( b_0 \wedge b_1 \unlhd b_2),$
\item $\neg ( b_1 \wedge b_2 \unlhd b_0),$
\item\label{tree3} $b_0 \wedge b_1 = b_1 \wedge b_2.$
\end{enumerate}

\noindent The quantifier-free type described in \eqref{tree3} corresponds to $\ov{a}_0$ in Definition \ref{restricted}, since it is the only quantifier-free type that is realized within $\ov{b}_0$ in $\B$ that is a preimage of $\ov{a}$.
\end{ex}

\begin{dfn}\label{restricted}  Fix finite tuples $\ov{a}, \ov{b}$ from $\A$ and an injection $f: \B \raw \A$.  Fix finite tuples $\ov{a}_0,\ov{b}_0$ from $\B$ such that $|\ov{a}| = |\ov{a}_0|$.
We say that $\ov{b}_0$ has the \textit{restricted inverse images under $f$ property for $\ov{a}$ witnessed by $\ov{a}_0$} if for any $\ov{c}_1$ in $\la f(\ov{b}_0) \ra_{\A}$ such that $\ov{c}_1 \sim_{\A} \ov{a}$, 
%CHANGED, affirmative!
$\ov{c}_0:=f^{-1}(\ov{c}_1) \subseteq \ov{b}_0$
%fix JUST THE TUPLE
and
$\ov{c}_0 \sim_{\B} \ov{a}_0$.
%IF it is in the range.
\end{dfn}

%\begin{rmk}  In the case that $\A$ is locally finite, we may assume that $\ov{a},\ov{b}$ in Definition \ref{restricted} above enumerate finitely generated substructures from $\A$, and replace $\sim_{\A}$ with $\cong_{\A}$.
%\end{rmk}

\begin{prop}\label{prop_fcn_restricted} Let $\A$ and $\B$ be any structures with a semi-retraction $(g,f)$ between $\A$ and $\B$.
%if $\A$ is locally finite and $\A$ is a semi-retract of $\B$, then 
For any finite tuples $\ov{a}, \ov{b}$ \textit{enumerating substructures of $\A$}, %$\ov{b}_0$
%has the restricted inverse images under $f$ property for $\ov{a}$ witnessed by $g(\ov{a})$ and
%
for any $\ov{b}_0' \sim_\B g(\ov{b})$,
%if $\ov{b}_0'$ enumerates a substructure of $\B$, then
%
$\ov{b}_0'$ has the restricted inverse images under $f$ property for $\ov{a}$ witnessed by $g(\ov{a})$.
\end{prop}

\begin{proof} Fix structures $\A, \B$ and a pair of maps $g: \A \raw \B$ and $f: \B \raw \A$ witnessing that $\A$ is a semi-retract of $\B$.  Fix finite tuples $\ov{a},\ov{b}$ enumerating substructures of $\A$ and let $\ov{a}_0:=g(\ov{a}), \ov{b}_0:=g(\ov{b})$, $\ov{a}_1:=f(\ov{a}_0)$, and $\ov{b}_1:=f(\ov{b}_0)$.  Let $n:=|\ov{a}|$.  

%IS THIS still useful?
%Let $\{\ov{d}_i\}_{i < \alpha}$ be a list of length-$n$ pairwise-non-$\sim_{\B}$ tuples from $\B$ such that $f(\ov{a}_i) \sim_{\A} \ov{a}$ for every $i<\alpha$, and for any length-$n$ tuple $\ov{c}$ from $\B$ such that $f(\ov{c}) \sim_{\A} \ov{a}$, there exists some $i <\alpha$ such that $\ov{c} \sim_{\B} \ov{d}_i$.   Such a list $\{\ov{d}_i\}$ is well defined by the qftp-respecting property of semi-retractions.  We may assume that $\ov{d}_0=\ov{a}_0$. 
 
%%%
Fix a tuple $\ov{c}_1$ in $\la \ov{b}_1 \ra_{\A}$ such that:
%$f^{-1}(\ov{c}) \sim_{\B} \ov{a}_0$.
\setcounter{equation}{0}
\begin{eqnarray}\label{0}
\ov{c}_1 \sim_\A \ov{a}.  
\end{eqnarray}

\noindent We will argue that $\ov{c}_0:=f^{-1}(\ov{c}_1)$ is in $\la \ov{b}_0 \ra_\B$: %restricted to $\ov{b}_0$:
Since $fg$ is an $\LA$-embedding 
%of structures 
and we assumed $\ov{b} = \la \ov{b} \ra_\A$, it must be that $\ov{b}_1 = \la \ov{b}_1 \ra_\A$. %(because structures are closed under the interpretation of function symbols.)  
Thus $\ov{c}_1 \subseteq \ov{b}_1$ so
there is some $\ov{c} \subseteq \ov{b}$ such that $g(\ov{c}) = \ov{c}_0$ and $f(\ov{c}_0) = \ov{c}_1$, and therefore $f^{-1}(\ov{c}_1)=g(\ov{c}) \subseteq \ov{b}_0$.
By the embedding property of semi-retractions, 
\begin{eqnarray}\label{0.5}
\ov{c} \sim_{\A} \ov{c}_1.
\end{eqnarray}

The equations \eqref{0} and \eqref{0.5} imply that:
\begin{eqnarray}\label{1}
\ov{c} \sim_{\A} \ov{a}.
\end{eqnarray}

\noindent Thus, by the qftp-respecting property of semi-retractions: 
\begin{eqnarray}\label{2}
\ov{c}_0 = g(\ov{c}) \sim_\B g(\ov{a}) = \ov{a}_0
\end{eqnarray}

\noindent By \eqref{2}, we may conclude that $\ov{c}_0=f^{-1}(\ov{c}_1) \sim_\B \ov{a}_0$, as desired.  %We summarize these facts in the diagram below.

%\begin{center}
%\begin{tikzcd}
%\ov{a}  \arrow[d,dash,dashed,"\eqref{1}"] \arrow[r,mapsto,"g"] &\ov{a}_0 \arrow[r,mapsto,"f"] %\arrow[d,dash,dashed,"\eqref{2}"]& \ov{a}_1 \arrow[d,mapsto]\\
%\ov{c} \arrow[rr,dash,dashed,bend right=50,"\eqref{0.5}"] 
%\arrow[r,mapsto,"g"]  &\ov{c}_0 \arrow[r,mapsto,"f"] & \ov{c}_1 \arrow[llu,dash,dashed,controls={+(-2,3) and +(-1,1)},"\eqref{0}"] 
%\end{tikzcd}
%\end{center}

To complete the proof, fix any $\ov{b}_0'$ from $\B$ such that 
\begin{eqnarray}\label{41}
\ov{b}_0' \sim_\B \ov{b}_0 .
\end{eqnarray}
\noindent Since $\ov{b}_0' \sim_\B \ov{b}_0$, $f(\ov{b}_0') \sim_\A f(\ov{b}_0)=\ov{b}_1$ and so $f(\ov{b}_0')$ inherits $\ov{b}_1$'s property of enumerating a substructure of $\A$, i.e. $\la f(\ov{b}_0') \ra_{\A}=f(\ov{b}_0')$.

\noindent Fix any $\ov{e}_1 \subseteq \ov{b}_1'=\la f(\ov{b}_0') \ra_{\A}=f(\ov{b}_0')$ such that $\ov{e}_1 \sim_\A \ov{a}$.  Then there exists $\ov{e}_0:=f^{-1}(\ov{e}_1) \subseteq \ov{b}_0'$. 
%
%By the definition of the $\{\ov{d}_i\}_{i<\alpha}$ there exists some $j<\alpha$ such that $\ov{e}_0 \sim_\B \ov{d}_j$.
The similarity \eqref{41} guarantees the existence of some $\ov{e}_0' \subseteq \ov{b}_0$ (on the same coordinates as $\ov{e}_0 \subseteq \ov{b}_0'$) such that $\ov{e}_0' \sim_\B \ov{e}_0$.
%Since this $\ov{e}_0' \sim_\B \ov{d}_j$, for some $j<\alpha$, it must be that $j=0$.  
By the qftp-respecting property for semi-retractions, $f(\ov{e}_0') \sim_\A f(\ov{e}_0) = \ov{e}_1 \sim_\A \ov{a}$, and we just argued that $\ov{b}_0$ has the restricted inverse images under $f$ property for $\ov{a}$ witnessed by $\ov{a}_0$, so $\ov{e}_0' \sim_\B \ov{a}_0$,
thus $\ov{e}_0 \sim_\B \ov{a}_0$, as desired.
%for any $\ov{c}_1$ in $\la f(\ov{b}_0) \ra_{\A}$ such that $\ov{c}_1 \sim_{\A} \ov{a}$, $f^{-1}(\ov{c}_1) \sim_{\B} \ov{a}_0$.
\end{proof}

Now we are ready to adapt Proposition \ref{prop_fcn_restricted} to the case of relational signatures.

%%%%%%%%%%%%%%PROOF OF
\begin{proof}[Proof of Proposition \ref{prop_rlnl_restricted}]   Let $\A, \B, f, g, A, B, A_0, B_0, A_1$ be as in the statement.  Let $\ov{a}, \ov{b}$ be enumerations of $A, B$, respectively, and let $\ov{a}_0 := g(\ov{a}), \ov{b}_0:=g(\ov{b}), \ov{a}_1:=f(\ov{a}_0)$.  Clearly $\ov{a}_0, \ov{b}_0, \ov{a}_1$ enumerate $A_0, B_0, A_1$, respectively, and $|\ov{a}|=|\ov{a}_0|$.  Fix any $B_0' \cong_{\LB} B_0$ and fix an isomorphism $\sigma : B_0 \raw B_0'$ and let $\ov{b}_0' :=\sigma(\ov{b}_0)$.  By Proposition \ref{prop_fcn_restricted}, $\ov{b}_0'$ has the restricted inverse images under $f$ property for $\ov{a}$ witnessed by $\ov{a}_0$.  Clearly, this implies that $B_0'$ has the relational restricted inverse images under $f$ property for $A$ witnessed by $A_0$.
\end{proof}

Having pointed out this technical property, we can deduce the relational case with little machinery.

%%%%%%%%%%%%%%PROOF OF
\begin{proof}[Proof of Theorem \ref{fmla_free_relational}] Fix structures $\A, \B$ in relational signatures and $A\in\age(\A)$.
%Assume that $\age(\B)$ consists of rigid elements.
Fix the pair of maps $g: \A \raw \B$ and $f: \B \raw \A$ witnessing  that $\A$ is a semi-retract of $\B$.
Assume that $A_0:=g(A)\in\age(\B)$ has Ramsey degree $d$ in $\age(\B)$. Note that $A_0$ is a substructure of $\B$ since the signature of $\B$ is relational.%(we will show $\A$ has RP.)

Let $B\in\age(\A)$ and let $c: {\A \choose A} \raw k$ be a coloring.
%Fix finite structures $A, B \subseteq \A$ and a coloring $c: {\A \choose A} \raw 2$. 
%
Denote $B_0:=g(B), A_1:= f(A_0), B_1:=f(B_0)$.  %Since the signatures are relational, $A_0$ and $B_0$ are finite substructures of $\B$. 
%and $A, B, A_1, B_1$ are finite substructures of $\A$.  
%Let $\ov{b}_0, \ov{a}, \ov{a}_0$ be any enumerations of $B_0, A, A_0$.
%
Define an induced coloring $c_0 : {\B \choose A_0} \raw k$ 
by $c_0(A_0'):=c(f(A_0'))$.
This coloring is well-defined by the qftp-preserving property of semi-retractions.

Since $d(A_0,\age(\B))=d$, there exists a copy $B_0'$ of $B_0$ such that $c_0$ takes at most $d$ colors on $B_0'$. %Thus, there exists $d<2$ such that $c_0(A_0')=d$ for all $A_0' \cong_{\LB} A_0$ in $B_0'$.
We will argue that $c$ takes at most $d$ colors on $B_1':=f(B_0') \subseteq \A$. 
%is a copy of $B$ homogeneous for the coloring $c$.

%First, note that since 
As $B_0' \cong_{\LB} B_0$, we have that $B_1' \cong_{\LA} B_1 (\cong_{\LA} B)$ by the qftp-respecting property of semi-retractions.
Fix any copy $A_1'$ of $A$ in $B_1'$, and let $A_0':=f^{-1}(A_1') \subseteq B_0'$.
By Proposition \ref{prop_rlnl_restricted}, $B_0'$ has the relational restricted inverse images under $f$ property for $A$ witnessed by $A_0$, thus, $A_0' \cong_{\LB} A_0$.  But then $A_0'$ is a copy of $A_0$ in $B_0'$, so $c_0(A_0')=c(f(A_0'))=c(A_1')$.
This proves the claim.
\end{proof}

%%%%%%%%%%%%%%%%%%%%%%%%%%%%%%%%%%%%%%%%%%%%%%%%%%%%%%%%%%
\section{Semi-retractions and categorical notions}\label{category}
In the next two sections, we point out similarities between semi-retractions and two category theoretic notions -- pre-adjunctions and retractions.

\subsection{Pre-adjunctions}\label{pread} 
Dragan Ma\v{s}ulovi\'c realized in \cite{mas18} that a categorical notion of pre-adjunction was implicitly used in coding one Ramsey problem into another in the book \cite{pr13} by Pr\"oml. The notions of pre-adjunction and semi-retraction appear to be closely related in the usual setting of classes of finitely-generated structures. We show that every semi-retraction defines a pre-adjunction and that in some instances pre-adjunctions define semi-retractions. For a category $\C$, we denote by $\obj(\C)$ its objects and by $\hom_\C(A,B)$ the collection of morphisms between objects $A, B \in \obj(\C)$.

\begin{dfn} Let $\C$ and $\D$ be categories and let $F:\obj(\D)\to\obj(\C)$ and \\$G:\obj(\C)\to \obj(\D)$ be maps on objects. We say that $(F,G)$ is a \emph{pre-adjunction} if for every $A\in\obj(\D)$ and $C\in\obj(\C)$ we have a map
	\[
	\Phi_{A,C} : \hom_\C(F(A),C)\to \hom_\D(A,G(C)),
	\]
	such that 
	\[\forall A,B\in \obj(\D)\;  \forall C\in\obj(\C)\;  \forall v\in \hom_\D(A,B)\;  \forall \psi\in \hom_\C(F(B),C)\]
\[
\exists w\in\hom_\C(F(A),F(B))\textrm{~such that~} \Phi_{A,C}(\psi\circ w)=\Phi_{B,C}(\psi)\circ v.
\]
\end{dfn}

We state a version of Ma\v{s}ulovi\'c's result restricted to our setting.

\begin{thm}[\cite{mas18}]\label{masresult}  Let $\C$ and $\D$ be categories of finite structures with embeddings as morphisms.  Assume that $F: \obj(\D) \rightleftarrows \obj(\C) : G$ is a pre-adjunction and that $\C$ has the Ramsey property for embeddings.  Then $\D$ has the Ramsey property for embeddings.
\end{thm}

\subsection{Pre-adjunctions from semi-retractions}
\begin{thm}\label{adj}
	Any semi-retraction $(g,f)$ between $\A$ and $\B$ defines a pre-adjunction between the categories of finite tuples of $\A$ and $\B,$ respectively, with qftp-preserving injections. 
\end{thm}

\begin{proof}
	Let $\D$ be all finite tuples of $\A$ with qftp-preserving injections as morphisms and let $\C$ be all finite tuples of $\B$ with qftp-preserving injections as morphisms. Define \\$F:\obj(\D)\to \obj(\C)$ by $F(\ov{a})=g(\ov{a})$ and $G:\obj(\C)\to \obj(\D)$ by $G(\ov{c})=f(\ov{c}),$ and we simply let $\Phi_{\ov{a},\ov{c}}:\hom_{\C}(F(\ov{a}),\ov{c})\to \hom_{\D}(\ov{a},G(\ov{c}))$ be defined by $\psi\mapsto f(\psi)\circ f\circ g.$ 
	
	Suppose that $\ov{a},\ov{b}\in\obj(\D)$, $\ov{c}\in\obj(\C)$, $\psi\in\hom_{\C}(g(\ov{b}),\ov{c}),$ and $v\in\hom_{\D}(\ov{a},\ov{b}).$ Let $w\in\hom_{\C}(g(\ov{a}),g(\ov{b}))$ be equal to $g(v)$. Then 
	\begin{eqnarray*}
		\Phi_{\ov{b},\ov{c}}(\psi)\circ v &=& f(\psi)\circ f\circ g\circ v  \\
		&=& f(\psi)\circ f\circ g(v)\circ g  \\
		&=& f(\psi)\circ f(g(v))\circ f\circ g  \\
		&=&  f(\psi\circ g(v))\circ f\circ g= f(\psi\circ w)\circ f\circ g \\
%uses injectivity of f?
		&=&	\Phi_{\ov{a},\ov{c}}(\psi\circ w).		
	\end{eqnarray*}
	
	\begin{multicols}{3}
		\begin{center}
			\begin{tikzcd}[row sep=huge,column sep=huge]
				\ov{a}  %\arrow[d,dash,dashed,"\eqref{1}"]
				\arrow[r,"v"] \arrow[dr,"\Phi_{\ov{a},\ov{c}}(\psi\circ w)"] &\ov{b} %\arrow[r,mapsto,"f"] 
				\arrow[d,dashed, "\Phi_{\ov{b},\ov{c}}(\psi)=f(\psi)"]& %\ov{a}_1 \arrow[d,mapsto]
				\\
				{} & %\ov{c} \arrow[rr,dash,dashed,bend right=50,"\eqref{0.5}"] 
				%\arrow[r,mapsto,"g"]  &
				G(\ov{c}) %\arrow[r,mapsto,"f"]  & \ov{c}_1 \arrow[llu,dash,dashed,controls={+(-2,3) and +(-1,1)},"\eqref{0}"] 
			\end{tikzcd}
		\end{center}
		\columnbreak
		\begin{center}
			\begin{tikzcd}[row sep=huge,column sep=huge]
				g(\ov{a}) \arrow[r,dashed, "w=g(v)"] & g(\ov{b}) \arrow[d,"\psi"]\\
				{} & \ov{c}
			\end{tikzcd}	
		\end{center}
		\columnbreak
		\begin{center}
			\begin{tikzcd}[row sep=huge,column sep=huge]
				fg(\ov{a}) \arrow[r,dashed, "f(w)=fg(v)"] \arrow[dr,"f(\psi\circ g(v))"]& fg(\ov{b}) \arrow[d,"f(\psi)"]\\
				{} & f(\ov{c})
			\end{tikzcd}	
		\end{center}
	\end{multicols}
Therefore $(F,G)$ is a pre-adjunction. 
\end{proof}

\begin{thm}\label{preadjage}
	Let $\A$ and $\B$ be locally finite and let $(g,f)$ be a semi-retraction between $\A$ and $\B$. Then there is a pre-adjunction between $\age(\A)$ and $\age(\B)$ with embeddings as morphisms. 
\end{thm}

\begin{proof}
	The proof goes along the same lines as the proof of Theorem \ref{adj}.
	Let $g:\A\raw\B$ and $f:\B\raw\A$ be a semi-retraction and let $\C$ be $\age(\B)$ with embeddings and let  $\D$ be $\age(\A)$ with embeddings. Define $F: \text{Obj}(\D)\raw \text{Obj}(\C)$ by $F(A)=\left<g(A)\right>_{\B}$ and $G:\obj(\C)\raw\obj(\D)$ by $G(C)=\left<f(C)\right>_{\A}.$ For any $A,B\in \obj(\D)$ and embedding $v:A\to B,$ denote by $g(v)'$ the unique extension of $g(v)$ to $\left<g(A)\right>_{\B}$, and similarly for $C,D\in\obj(\C)$ and embedding $w:C\to D$ we define $f(w)'$ to be the unique extension of $f(w)$ to $\left< f(C)\right>_{\A}.$ Finally, define $\Phi_{B,C}:\hom_{\C}(\left<g(B)\right>_{\B},C)\to \hom_{\D}(B, \left<f(C)\right>_{\A})$ by $\psi\mapsto f(\psi\restriction g(B))\circ f\circ g$. By analogous diagram chasing as in the proof of Theorem \ref{adj}, we can verify that we obtain a pre-adjunction.
	
	\begin{multicols}{2}
		\begin{center}
			\begin{tikzcd}[row sep=huge,column sep=huge]
				A  %\arrow[d,dash,dashed,"\eqref{1}"]
				\arrow[r,"v"] \arrow[dr,"\Phi_{A,C}(\psi\circ w)"] & B %\arrow[r,mapsto,"f"] 
				\arrow[d,dashed, "\Phi_{B,C}(\psi)"]& %\ov{a}_1 \arrow[d,mapsto]
				\\
				{} & %\ov{c} \arrow[rr,dash,dashed,bend right=50,"\eqref{0.5}"] 
				%\arrow[r,mapsto,"g"]  &
				\la f(C) \ra_\A %\arrow[r,mapsto,"f"]  & \ov{c}_1 \arrow[llu,dash,dashed,controls={+(-2,3) and +(-1,1)},"\eqref{0}"] 
			\end{tikzcd}
		\end{center}
		\columnbreak
		\begin{center}
			\begin{tikzcd}[row sep=huge,column sep=huge]
				\la g(A) \ra_\B \arrow[r,dashed, "w=g(v)'"] & \la g(B) \ra_\B \arrow[d,"\psi"]\\
				{} & C
			\end{tikzcd}	
		\end{center}
	\end{multicols}
\end{proof}

\subsection{Semi-retractions from pre-adjunctions}
We now consider the reverse direction - building semi-retractions out of pre-adjunctions.

Suppose that $\A$ and $\B$ are countable structures in relational signatures. Let $\D$ be $\age(\A)$ with embeddings as morphisms and let $\C$ be $\age(\B)$ with embeddings as morphisms. 

\begin{thm}
	Suppose that $(F,G)$ is a pre-adjunction between $\D$ and $\C$ and assume that $F$ and $G$ preserve cardinality. Then there is a semi-retraction between $\A$ and a structure $\B'$ whose age is contained in $\C$. 
\end{thm}

\begin{proof}
Let $\A=\bigcup_{i=1}^\infty A_i$ be an increasing enumeration of $\A,$ where each $||A_i||=i.$ For any $e_i:A_i\subset A_{i+1}$ we consider the identity morphis, $\id:F(A_{i+1})\raw F(A_{i+1})$ to obtain $w_i:F(A_i)\raw F(A_{i+1})$ by the definition of pre-adjunction satisfying the following:
$$\Phi_{A_i,F({A_{i+1}})}(w_i)=\Phi_{A_i,F({A_{i+1}})}(w_i\circ \id_{F(A_{i+1})})=\Phi_{A_{i+1},F({A_{i+1}})}(\id_{F(A_{i+1}}))\circ e_i.$$

We will define $g:\A\to \B'$ for some $\B'$ with $\age$ included in $\age(\B)$ by $g(A_1)=F(A_1)$ and $g(A_{i+1}\setminus A_i)=F(A_{i+1})\setminus w_i(F(A_i))$. %[This way $g$ will be onto.]

\begin{multicols}{2}
		\begin{tikzcd}[row sep=huge,column sep=huge]
			A_i  %\arrow[d,dash,dashed,"\eqref{1}"]
			\arrow[r,"e_i"] \arrow[d,"a_i=\Phi_{A_i,F(A_i)}(\id_{F(A_i)})"] & A_{i+1} %\arrow[r,mapsto,"f"] 
			\arrow[d, "a_{i+1}"]& %\ov{a}_1 \arrow[d,mapsto]
			\\
			GF(A_i)  %\ov{c} 
			\arrow[r,"b_i"] &
			%\arrow[r,mapsto,"g"]  &
			GF(A_{i+1}) %\arrow[r,mapsto,"f"]  & \ov{c}_1 \arrow[llu,dash,dashed,controls={+(-2,3) and +(-1,1)},"\eqref{0}"] 
		\end{tikzcd}
	\columnbreak
	\begin{center}
		\begin{tikzcd}[row sep=huge,column sep=huge]
			F(A_i) \arrow[r,dashed, "w_i"] & F(A_{i+1}) \arrow[d,"\id_{F(A_{i+1})}"]\\
			{} & F(A_{i+1})
		\end{tikzcd}	
	\end{center}
\end{multicols}

Since $F$ and $G$ are cardinality preserving, we have for  every $A_i$ that $a_i:=\Phi_{A_i, F(A_i)}(\id_{F(A_i)})$ is an isomorphism from $A_i$ to $GF(A_i).$ It follows that $b_i=a_{i+1}\circ e_i\circ a_i^{-1}$ is an embedding from $GF(A_i)$ into $GF(A_{i+1})$ and that $\bigcup_{i=1}^{\infty}(GF(A_i),b_i)\cong \A$ via $\bigcup_{i=1}^{\infty} a_i$.% [verify everything commutes]

We define $f:\B'\to \A$ by $f(b)=a$ iff $a_i(g^{-1}(b))=a.$
%\begin{center}
%	\begin{tikzcd}[row sep=huge,column sep=huge]
%	A_i  %\arrow[d,dash,dashed,"\eqref{1}"]
%	\arrow[r,"e_i"] \arrow[d,"g"]\arrow[dd, controls={+(-1,0) and +(-1,1)},"a_i"] & A_{i+1} %\arrow[r,mapsto,"f"] 
%	\arrow[d, "g"] \arrow[dd,controls={+(1,0) and +(1,1)},   "a_{i+1}"] & %\ov{a}_1 \arrow[d,mapsto]
%	\\
%	F(A_i)  %\ov{c} 
%	\arrow[r,"b_i"] 
%	\arrow[d,mapsto,"f"]  &
%	F(A_{i+1}) \arrow[d,mapsto,"f"]\\  %& \ov{c}_1 \arrow[llu,dash,dashed,controls={+(-2,3) and +(-1,1)},"\eqref{0}"] 
%	 GF(A_i) \arrow[r,"w_i"] &GF(A_{i+1})
%\end{tikzcd}
%\end{center}
%We get that $f,g$ are surjective, so $\A$ is a reduct of $\B'$ by Remark \ref{rmk_simple}.
\end{proof}

%In Section \ref{counterexample}, we were able to obtain a semi-retraction out of a pre-adjunction between finite ordered graphs and finite ordered Boolean algebras from \cite{mas18}. In order to understand the relative strength of these two notions to transfer Ramsey properties, we pose the following question.

\begin{qu}
	Let $\A$ and $\B$ be (locally finite) structures and suppose that there is a pre-adjunction between $\age(\A)$ and $\age(\B)$ with embeddings as morphisms. Under which conditions is there a semi-retraction between $\A$ and $\B$? 
\end{qu}

%\section{Semi-Retractions, Retractions and Interpretations}
\section{Concluding Remarks}\label{conclusion}

In this section we gather together some concepts and prior work related to semi-retractions.  From a model-theoretic perspective, it is natural to ask what is the relationship of semi-retractions to interpretations, and we approach this question in Section \ref{7.1}.  In Section \ref{7.3} we cite work that has addressed the question of when interpretations preserve RP.  

The notion of a retraction is defined generally for categories, which we quote from \cite{berg}.
\begin{dfn} Given a category $\CC$ and morphisms $f \in \CC(X,Y)$ and $g \in \CC(Y,X)$, if $fg = \id_{Y}$, i.e.
$$Y \xrightarrow{g} X \xrightarrow{f} Y \text{\ and \ } Y \xrightarrow{fg=\id} Y$$
%and
%$$Y \xrightarrow{fg=\id} Y$$
we say that the pair of maps $(g,f)$ %[my choice of order] 
is a \emph{retraction} of $X$ onto $Y$ and that $Y$ is the \emph{retract} of $X$ (via $f$ and $g$).
\end{dfn}
\noindent The term ``retraction" in Definition \ref{AZ} implies the existence of a retraction in the category of topological groups.
In Section \ref{7.2} we indicate a retraction that exists in the category of products of types spaces $(S^\A_n(\emptyset))_{n<\omega}$ for structures $\A$ that have the property \texttt{rqe}, which retraction is induced by any semi-retraction.

The fact that pairs of interpretations (under certain conditions) induce a retraction in a distinct category from the category associated to semi-retractions suggests that these are different concepts.  We also give an example of an interpretation map that is not qftp-respecting in Example \ref{not_sr}.

%

%Our notion of retraction was named as such because of a superficial resemblance to a different notion of retraction given in \cite{ahzi86} which we will explain below. 
%Given a structure $\A$ we will use the roman letter $A$ to denote the underlying set of $\A$ where this adds clarity.

\subsection{Interpretations and the Ramsey property}\label{7.1} The results presented in this subsection are well-known.  

\begin{dfn}\label{interp}
Given two structures $\A, \B$ a \emph{(finitary) interpretation} of $\B$ in $\A$ denoted by $f: \A \rightsquigarrow \B$ is a surjective function $f: U \raw |\B|$ where $U \subseteq {|\A|}^n$ is 0-definable, $n<\omega$, and for every integer $m$ and $0$-definable $m$-ary relation $R$ in $\B$ (where we include the relation $x_0=x_1$), the set
$$\{(\ov{a}_0,\ldots,\ov{a}_{m-1}) : \B \vDash R(f(\ov{a}_0),\ldots,f(\ov{a}_{m-1})) \} \subseteq |\A|^{n\cdot m}$$
is $0$-definable in $\A$.  In other words, there exists a formula
$\varphi_R(\ov{x}_0,\ldots,\ov{x}_{m-1})$ in the signature of $\A$ such that for all $\ov{a}_i \in U$, for all $i<m$:
$$\A \vDash \varphi_R(\ov{a}_0,\ldots,\ov{a}_{m-1}) \Leftrightarrow \B \vDash R(f(\ov{a}_0),\ldots,f(\ov{a}_{m-1}))$$
\end{dfn}

\begin{rmk} In Definition \ref{interp}:
\begin{itemize}
%\item it suffices to replace the definable relations $R$ with only basic atomic relations, or with only quantifier-free definable relations.
\item We may define $E \subseteq U \times U$ to hold of $(\ov{a}_1,\ov{a}_2)$ if and only if $\B \vDash f(\ov{a}_1) = f(\ov{a}_2)$, in other words, $E:=\varphi_=$.  Then $\B$ is isomorphic to the \emph{induced structure} on $U/E$ (reducted to the signature $\{\varphi_R\}_R$), what is called a relativised reduct (of a definitional expansion) of $\A^{\textrm{eq}}$ in Theorem 5.3.1 of \cite{ho93}.
\item If $\B$ is interpreted in $\A$ such that the $E$-classes in $\A$ (as described above) are singletons, then $\B$ is a \emph{(relativised) reduct} of $\A$.
\end{itemize}
\end{rmk}

In Example \ref{not_sr} we give an example of an interpretation map whose inverse is not qftp-respecting, and thus does not constitute half of a semi-retraction pair.
%%%%%%%%%%%%%%%%%%%%%%
\begin{ex}\label{not_sr}
Let $\A := (\Z,s)$, $\B := (\Z + \Z,s)$, and $f: \A \raw \B$ be the bijection that maps the even numbers of $\A$ onto the ``first'' copy of $\Z$ in $\B$ and the odd numbers of $\A$ onto the ``second'' copy of $\Z$ in $\B$, in the natural way.
%Exploit the fact that $\A$ is not the prime model, so has a non-isolated type: 
Let $p \in S_2^\B(\emptyset)$ be the 2-type $p(x,y):=\{ s^n(x) \neq y: n < \omega\}$.
%\underline{$g$ is not qftp-respecting}: 
Let $g : \B \raw \A$ be defined as the inverse map, $g = f^{-1}$.  The function $g$ maps realizations $(x,y)$ of the type $p$ in $\B$ onto realizations of various $2$-types in $\A$ (namely, all pairs of odd distances) thus $g$ is not qftp-respecting.
%\underline{$g$ interprets $\A$ in $\B$}: 
To see that $f$ interprets $\B$ in $\A$,
we just need to look at the atomic formula $s(x_0) = x_1$.  For all $(a_0,a_1) \in |\A|$,
$$\B \vDash s(f(a_0)) = f(a_1) \Leftrightarrow \A \vDash s(s(a_0)) = a_1 .$$
\end{ex}
%%%%%%%%%%%%%%%%%%%%
%%%%removed

\begin{dfn}\label{homotopic} An interpretation $h: W \rightsquigarrow \A$ for $W \subseteq |\A|^k$ is \emph{homotopic to the identity interpretation on $\A$} if the set 
$\{(\ov{a},b) : h(\ov{a})=b \} \subseteq W \times |\A|$ is 0-definable in $\A$.
\end{dfn}

\begin{rmk} Note that the interpretation $h$ in Definition \ref{homotopic} gives an isomorphism between the induced structure on $W/E$ and $\A$ that is definable in $\A$.
\end{rmk}

\begin{dfn}[\cite{ahzi86}]\label{AZ} Given countable, $\aleph_0$-categorical structures $\A$ and $\B$, $\A$ is a \emph{retraction} of $\B$ if there exist interpretations $f: \A \rightsquigarrow \B$  $g: \B \rightsquigarrow \A$ such that $g \circ f$ is homotopic to the identity interpretation on $\A$.
%%ctbl structures in ctbl lang
\end{dfn}

\begin{dfn}
Given countable $\aleph_0$-categorical structures $\A, \B$, if $\A$ and $\B$ are retractions of one another, then we say they are \emph{bi-interpretable}.
\end{dfn}

In \cite{ahzi86}, Ahlbrandt and Ziegler introduce the $\textbf{Aut}$ functor between the category of countable $\aleph_0$-categorical structures in a countable signature with interpretations as maps and the category of topological groups. The functor $\textbf{Aut}$ associates to each interpretation $f: \A \rightsquigarrow \B$ a continuous homomorphism $\textbf{Aut} f: \Aut(\A) \raw \Aut(\B)$ in the natural way.  
%Moreover, a continuous homomorphism $\varphi: \Aut(\A) \raw \Aut(\B)$ is of the form $\textbf{Aut} f$ if Thus we have the following.
The following result, attributed to T.~Coquand, follows from the work in \cite{ahzi86}.

\begin{thm}[T.~Coquand]\label{coquand}  Given countable $\aleph_0$-categorical structures $\A$ and $\B$,
$\A$ is a retraction of $\B$ iff there are continuous homomorphisms

$$\Aut(\A)  \xrightarrow{\varphi} \Aut(\B)  \xrightarrow{\psi} \Aut(\A)$$
such that $\psi \circ \varphi = 1$.
\end{thm}

\begin{cor}[\cite{ahzi86}]\label{biint} Countable $\aleph_0$-categorical structures
$\A, \B$ are bi-interpretable if and only if their automorphism groups are isomorphic as topological groups.
\end{cor}

%It follows from the Theorem \ref{coquand} above that 

In the preliminaries we introduced the notion of a definable set.  Here we introduce a generalization.
\begin{dfn}  Fix a structure $\A$, and an integer $n \geq 1$.
A set $X \subseteq {|\A|}^n$ is \emph{quasidefinable} if it is the union of $\Aut(\A)$-orbits of ${|\A|}^n$, where the action of $\sigma \in \Aut(\A)$ on ${|\A|}^n$ is given by $(a_0,\ldots,a_{n-1}) \mapsto (\sigma(a_0),\ldots,\sigma(a_{n-1}))$.
\end{dfn}

\begin{rmk}  Using Scott sentences one can show that any quasidefinable subset of a countable structure $\A$ is definable by an $L_{\omega_1,\omega}$ formula with no parameters.  If $\A$ is countable and $\aleph_0$-categorical, then quasidefinable sets are in fact 0-definable (see \cite{kama94}).
\end{rmk}

%%bigger
\begin{rmk} In Definition \ref{interp}, if ``0-definable'' is replaced with ``quasidefinable''  and the map $f$ is replaced with a countable collection of maps $f_i$ such that equality between $f_i(\ov{u})$ and $f_j(\ov{v})$ is quasidefinable in $\A$ (just as equality between $f(\ov{u})$ and $f(\ov{v})$ must be definable in $\A$ in a finitary interpretation) we may call this an \emph{infinitary interpretation}, see the introductory chapter of \cite{kama94}, ``Models and Groups.''
\end{rmk}

An account due to Kaye in \cite{kama94}, ``Models and Groups'' expands on Corollary \ref{biint} as follows:

\begin{thm}[\cite{kama94}] If $\A, \B$ are countable structures, $\Aut(\A) \cong \Aut(\B)$ as topological groups if and only if $\A, \B$ are infinitarily bi-interpretable.
\end{thm}

%The following is an easy proposition, but it is worth noting.

It is satisfying to see that the following proposition and corollary follow immediately from \cite{kpt05} and \cite{zu16}.

\begin{prop}\label{retract} Let $\A, \B$ be countable $\omega$-homogeneous structures, and suppose that $\A$ is a retraction of $\B$ in the sense of Definition \ref{AZ}.  If $\B$ has finite Ramsey degrees for embeddings, then $\A$ has finite Ramsey degrees for embeddings.
\end{prop}

\begin{proof} Since $\B$ is $\omega$-homogeneous, if $\B$ has finite Ramsey degrees for embeddings then the universal minimal flow  $M(\Aut(\B))$ is metrizable by the reverse direction of Theorem \ref{zucker}.  Let $K$ be the kernel of the quotient $\psi: \Aut(\B) \raw \Aut(\A)$. Any $\Aut(\A)$-flow on $X$ induces an $\Aut(\B)$-flow on $X$ by $fx=(Kf)x$, where $\Aut(\A)$ is identified with the quotient group $\{Kf:f\in \Aut(\B)\}$. If $M$ is a non-metrizable minimal flow of $\Aut(\A),$ it induces a non-metrizable minimal $\Aut(\B)$-flow as above, which is a contradiction. 
We may conclude that $M(\Aut(\A))$ is metrizable, and thus that $\A$ has finite Ramsey degrees  for embeddings by Theorem \ref{zucker}.
\end{proof}

%Is this too colloquial?

\begin{cor}  If $\A, \B$ are countable $\omega$-homogeneous infinitarily bi-interpretable structures, then $\A$ has finite Ramsey degrees (for embeddings) if and only if $\B$ has finite Ramsey degrees (for embeddings).
\end{cor}

\begin{rmk}
Note that in the proof of Proposition \ref{retract} we only needed that $\Aut(\A)$ is a quotient of $\Aut(\B).$
\end{rmk}

%%complete:
%\begin{ex} Here we give a specific example of a retraction that illustrates Proposition \ref{retract} above.
%%it could be equality or it could be...something from the paper or Ahlbrandt?
%\end{ex}
%Notice that the equivalence relation is not generic on $\B$, nor can we be guaranteed that for any equivalence relation imposed on a structure $\B$ with the RP, will $\Aut(\B)$ necessarily be extremely amenable, nor will there necessarily be the retraction maps.
%%Lynn: Apr 03

%%%%
\subsection{Related Work}\label{7.3}

It is a natural question to ask under what conditions interpretations transfer the Ramsey property.  
%This question has been approached in many places, including some works that will be highlighted here.
Proposition \ref{bod} gives an example of this type of result.

%sub omega homog
\begin{prop}[Proposition 3.8 in \cite{bod15}]\label{bod}  Given a countable $\omega$-homogeneous $\aleph_0$-categorical structure $\Gamma$ with RP for embeddings, every structure $\M$ with a first-order interpretation in $\Gamma$ has an $\aleph_0$-categorical expansion $\N$ with RP for embeddings.  Furthermore, if $\Gamma$ is $\omega$-homogeneous in a finite relational signature, we can choose $\N$ to be $\omega$-homogeneous in a finite relational signature.
\end{prop}

In \cite{sokrel13}, classes of structures interpretable in a fixed class $\K$ of finite structures in some relational signature $L$ are defined and investigated with respect to the Ramsey property.    For example, $\EK$ is defined to be all structures $A$ in the signature $L \cup \{E\}$, where $E \notin L$ is interpreted as an equivalence relation. In the following, by an \emph{ordered expansion} of a structure we mean an expansion by a new relation symbol whose interpretation linearly orders the structure.  Suppose that $\K^*$ is a class of ordered expansions of structures in $\K$, (this context could easily be adapted to the case where $\K$ is ordered by a relation in $L$) and define $\CEKS$ to be all ordered expansions of structures $A \in \EK$ such that the ordering is \emph{convex} (meaning that each equivalence class is an interval with respect to the ordering) and the induced ordered structure $A/E \in \K^*$.  Theorem 4.5 of \cite{sokrel13} states that $\K^*$ has the Ramsey property if and only if $\CEKS$ has the Ramsey property.

The previous works provide an intriguing look into the relationship between the Ramsey property and interpretations, but do not completely settle the question.  Suppose $\A$ is interpreted in $\B$ by way of $\D$, $E$, where $\D$ is a definable subset of $\B$, $E$ is a definable equivalence relation on $\D$, and $\A$ is isomorphic to some reduct of $D / E$.  If $\age(\A) = \K$, and $\age(\B) = \K'$, then the interpretation demonstrates that a subset (the substructures of $D$) of a reduct of $\K'$ coincides with $\EK$.
%, where the new relation symbol $E$ is defined to be the definable relation $E$.
%structures who are in $\A$ mod $E'$ are themselves in $\K'$ to begin with, assuming they were in D.
%$\EK$ is a subset of a reduct of $\K'$
Thus, the Ramsey transfer given by a particular interpretation is related to the question of which reducts of a structure with the Ramsey property have the Ramsey property.
%This would be a hard question.

We have learned that in recent work, a notion weaker than a semi-retraction has been used in Lemma 3.14(iii) of \cite{kamsma2023positive}, which paper generalizes many of the Ramsey transfer results of \cite{kks14} to positive logic.

\subsection{Semi-Retractions and Retractions}\label{7.2}

%%complete:
%The motivation for this section is to compare our notion of retraction to the previously studied notion in Definition \ref{AZ}.  
In the case that $\A, \B$ are countable and $\aleph_0$-categorical, the type spaces are finite, and so a well-defined map between them is continuous.  Under additional assumptions, qftp-respecting maps induce  maps between  type spaces.
%supply

\begin{dfn}\label{rqe} Say that a structure $\A$ has property $\rqe$ (realized, quantifier-eliminable types) if
$$S_n^\A(\emptyset) = \{ \qftp(\ov{a}) : \ov{a} \in {|\A|}^n \}.$$
%textrm{~and $\ov{a}$ enumerates a finite substructure} \} = $$
\end{dfn}

\begin{obs} If $\A$ has $\rqe$ then every $n$-type over the empty set in $\A$ is realized in $\A$.
\end{obs}

\begin{rmk} If $\A$ is in a finite relational signature and $\Th(\A)$ has  quantifier elimination, then complete quantifier-free types are complete types that are also finite types, so any type realized in an elementary extension is realized in $\A$.  Thus $\A$ has $\rqe$.

More generally, if $\A$ is an $\omega$-homogeneous uniformly locally finite structure in a finite signature, then $\Th(\A)$ is $\aleph_0$-categorical and has quantifier elimination as in Remark \ref{rmk_simple}, thus $\A$ has $\rqe$.
\end{rmk}
%did I write this?

%\begin{obs} It is observed in \cite{berg} that in the category $\CC = \mathbf{Set}$ 
%\begin{itemize}
%\item $g$ as above must be 1-1
%\item $gf: X \xrightarrow{\textrm{onto}} B$ where $\im g = B \subseteq X$
%\item $B \xrightarrow{gf=\id} B$
%\end{itemize}
%\end{obs}

\begin{dfn}\label{inducedmap} Given $\A, \B$ that have \texttt{rqe}, let $g: \A \raw \B$ be a qftp-respecting injection on the underlying sets.  We define $\theta_g: \sna \raw \snb$ to take $p \in \sna$ to $q \in \snb$ such that there exists $\ii \in {|\A|}^n$ satisfying $p(\ov{x})$ and $g(\ii) \in \B^n$ satisfies $q(\ov{x})$ (i.e. $q = \qftp^\B(g(\ii))$).
\end{dfn}

\begin{rmk}  Note that Definition \ref{inducedmap} is well-defined since for any $\ii, \ii'$ satisfying $p(\ov{x})$, $\ii \sim_\A \ii'$ and so $g(\ii) \sim_\B g(\ii')$ and thus $\qftp^\B(g(\ii))=\qftp^\B(g(\ii'))$.
\end{rmk}

\begin{obs}\label{cont} Given countable, $\aleph_0$-categorical $\A, \B$ that have \texttt{rqe}, let $g: \A \raw \B$ be a qftp-respecting injection on the underlying sets.  The map $\theta_g$ as in Definition \ref{inducedmap} is continuous.
\end{obs}

\begin{prop} Let $\A, \B$ be countable $\aleph_0$-categorical and $\rqe$.
If $\A$ is a semi-retract of $\B$ via $(g,f)$, then the pair of maps $(\theta_g,\theta_f)$ is a retraction of $\snb$ onto $\sna$ in the category $\CC$ of type spaces with continuous maps as morphisms.
\end{prop}

\begin{proof} By Observation \ref{cont}, we know that $(\theta_g, \theta_f)$ are maps in the category $\CC$.

By assumption, there exist qftp-respecting injections on the underlying sets
$A \xrightarrow{g} \B \xrightarrow{f} \A$
such that the composition is qftp-preserving
$\A \xrightarrow{fg} \A.$

Thus,
$$\sna \xrightarrow{\theta_g} \snb \xrightarrow{\theta_f} \sna$$
and
$$\sna \xrightarrow{\theta_f \cdot \theta_g = \id_{\sna}} \sna$$
which shows that the pair of maps is indeed a retraction in this category.
\end{proof}

\section{Appendix}

We start by providing proofs of some of the basic results quoted in the Preliminaries.

\begin{Cpct} Fix a signature $L$ and a locally finite $L$-structure $M$.  Let $\K:=\age(\M)$.  Then, for any finite substructure $A \subseteq \M$,
\begin{enumerate}
    \item\label{i} $d_e(A,\K) = d_e(A,\M)$, and
    \item\label{ii} $d(A,\K) = d(A,\M)$. 
\end{enumerate}
\end{Cpct}

\begin{proof} Fix $L$, $\M$, $\K$ and $A$ as in the assumptions.  Let $n := ||A||$.  First we prove \eqref{ii} and then we indicate how \eqref{i} follows by a similar argument.  
%Within this argument only, given finite tuples $\ov{x}, \ov{y}$ where we define $m:=|\ov{x}|$ and $k:=|\ov{y}|$, define $\ov{x} \subseteq \ov{y}$ to mean that for some $i_0<\ldots<i_{m-1} < k$, $x_j = y_{i_j}$, for all $j<m$.  Define $\ov{x} \subseteq^e \ov{y}$ to mean that for some permutation $\sigma: m \to m$, $x_{\sigma(j)} = y_{i_j}$, for all $j<m$.
%10.12

Since $\M$ is assumed to be locally finite with age $\K$, we may expand $L$ to a signature $L'$ that contains new predicates $p_C(\ov{x})$, for all $C \in \K$, and expand $\M$ to an $L'$-structure $\M'$ such that $\M' \vDash p_C(\ov{c}')$ if and only if $\ov{c}'$ is an enumeration of some structure $C'$ such that $C' \cong C$.  

Let $\Diag(\M')$ be the atomic diagram of $\M'$ in the signature $L'$ in variables $\{x_c : c \in \M\}$.  In other words, for any atomic $L'$-formula $\varphi(x_0,\ldots,x_{n-1})$, $\varphi(x_{c_0},\ldots,x_{c_{n-1}}) \in \Diag(\M')$ if and only if $\M' \vDash \varphi(c_0,\ldots,c_{n-1})$. 

We may further expand $L'$ to a signature $L^*$ that contains new $n$-ary predicate symbols $\{R_i(\ov{x}) : i < \omega\}$.  Define $\theta_{A,r}$ to be the $L^*$-sentence stating that the interpretations of $\{R_i\}_{i<r}$ form a partition on the copies of $A$ in $\M$:

$$\theta_{A,r}:= \forall \ov{x} \left( \left( p_A(\ov{x}) \leftrightarrow \bigvee_{i<r} R_i(\ov{x})\right) \wedge \bigwedge_{i \neq j < r} \neg (R_i(\ov{x}) \wedge R_j(\ov{x})) \wedge \bigwedge_{i < r, \sigma \in \textrm{Sym}(n)} (R_i(\ov{x}) \leftrightarrow R_i(\sigma(\ov{x})))\right).$$
%10.10

Define $\psi_{B,A,r,d}$ to be the $L^*$-sentence stating (in conjunction with $\theta_{A,r}$) that the $r$-coloring of copies of $A$ given by the interpretations of the $\{R_i\}_{i<r}$ achieves at least $d$ colors on copies of $A$ within any copy of $B$:

$$\psi_{B,A,r,d}:=\forall \ov{y} \left(p_B(\ov{y}) \rightarrow 
\bigvee_{\text{$\ov{x}_0,\ldots,\ov{x}_{d-1} \subseteq \ov{y}$;} \atop \text{$|\ov{x}_i|=n, i<d$}} 
\left[ \bigwedge_{i<d} p_A(\ov{x}_i)  \wedge \bigvee_{i_0<i_1<\ldots<i_{d-1}<r} \left( \bigwedge_{j<d} R_{i_j}(\ov{x}_j) \right) \right] \right).$$

Suppose $d(A,\K) \geq d$, for some integer $d \geq 1$.
By definition, there exists an integer $r \geq 2$ and a structure $B \in \K$ such that for all $C \in \K$, there is a coloring $c: {C \choose A} \rightarrow r$ such that for all $B' \subseteq C$ with the property that $B' \cong B$, $|c({B' \choose A})| \geq d$.  Thus, the following $L^*$-type $\Gamma$ is finitely satisfiable (in expansions of $\M'$):
$$\Gamma(\{x_c : c \in |\M|\}):= \Diag(\M') \cup \{\theta_{A,r} \wedge \psi_{B,A,r,d} \}.$$ 
%this was the typo
A realization of $\Gamma$ induces a coloring $c_0: {\M \choose A} \rightarrow r$ such that for any $B' \subseteq M$ with the property that $B' \cong B$, 
$|c_0({B' \choose A})| \geq d$.  Thus, we have shown that $d(A,\M) \geq d$.
%it can't be less than d

If $A$ does not have finite Ramsey degree in $\K$, then for all integers $d \geq 1$, $d(A,\K) \geq d$.  By the previous argument, $d(A,\M) \geq d$, for all integers $d \geq 1$, and so $A$ does not have finite small Ramsey degree in $\M$.

If $A$ does have finite Ramsey degree in $\K$, then it follows immediately that $A$ has finite small Ramsey degree in $\M$, and $d(A,\M) \leq d(A,\K)$.  By the argument above, we also have that $d(A,\M) \geq d(A,\K)$, and so $d(A,\M) = d(A,\K)$.

\vspace{.1in}

To see \eqref{i}, for each finite substructure $C \subseteq \M$ we fix an enumeration $\ov{c}$ of $C$ that we call its ``natural'' enumeration.
We then expand $L$ to a signature $L'$ that contains new predicates $p^e_C(\ov{x})$, for all $C \in \K$, and expand $\M$ to an $L'$-structure $\M'$ such that $\M' \vDash p^e_C(\ov{c}')$ if and only if $\ov{c}' \sim \ov{c}$, where $\ov{c}$ is the natural enumeration of $C$.

Then we expand by predicates $\{R_i(\ov{x}) : i < \omega\}$ and define

$$\theta^e_{A,r}:= \forall \ov{x} \left( p^e_A(\ov{x}) \leftrightarrow \bigvee_{i<r} R_i(\ov{x}) \wedge \bigwedge_{i \neq j < r} \neg (R_i(\ov{x}) \wedge R_j(\ov{x})) \right)$$

$$\psi^e_{B,A,r,d}:=\forall \ov{y} (p^e_B(\ov{y}) \rightarrow 
\bigvee_{\text{$\ov{x}_0,\ldots,\ov{x}_{d-1} \subseteq \ov{y}$;} \atop \text{$|\ov{x}_i|=n, i<d$}} 
\left[ \bigwedge_{i<d} p^e_A(\ov{x}_i)  \wedge \bigvee_{i_0<i_1<\ldots<i_{d-1}<r} \left( \bigwedge_{j<d} R_{i_j}(\ov{x}_j) \right) \right] ).$$

After this, the argument proceeds in the same manner.
\end{proof}

\begin{Duo} Fix a signature $L$ and a locally finite $L$-structure $M$.  Let $\K:=\age(\M)$.  Then, for any finite substructures $A, B \subseteq \M$,
$(A,B)$ is a Ramsey duo for $\K$ if and only if $(A,B)$ is a Ramsey duo for $\M$.
\end{Duo}

\begin{proof}
    If $(A,B)$ is a Ramsey duo for $\K$, it follows immediately that $(A,B)$ is a Ramsey duo for $\M$.  Suppose $(A,B)$ is not a Ramsey duo for $\K$.  Then there exists an integer $r \geq 2$ such that for all $C \in \K$, there exists a bad $r$-coloring $f_C: {C \choose A} \rightarrow r$.  As in Proposition \ref{compactness_emb_degrees}, we can write the type of a structure $\M'$ isomorphic to $\M$ with a bad $r$-coloring, i.e. such that for all $B' \cong B$ in $\M'$, $|c({B' \choose A})| \geq 2$.  This type is finitely satisfiable using the bad colorings $f_C$ and the fact that $\K = \age(\M)$.  Now the $L$-reduct of the realization of this type is isomorphic to $\M$, and the isomorphism induces a bad coloring on $\M$.
\end{proof}

%%%%%%%%%%%%%%%%%%%%%
%MT argument
%%%%%%%%%%%%%%%%%%%%%%

\subsection{An RP transfer argument from the perspective of substructures}\label{ModelTheoryFunctionalSection}

In this section, we develop an argument for Corollary \ref{cor_fmla_free} built on the ``restricted inverse images under $f$" property as defined in Definition \ref{restricted}.

%if \A is rigid and \b is ultrahomogeneous and g(B) < B_0, then we could probably make the argument work.

%{\color{red}Question: $f$ needs to be qftp-respecting on how big a subset $Y \subseteq \B$?  Could it just be qftp-respecting on $g(\A) \subseteq \B$?  Check.}

\begin{FmlaFree} Let $\A, \B$ be structures  and suppose $(g,f)$ is a semi-retraction between $\A$ and $\B$.   
Suppose that $\ov{a}, \ov{b}$ are finite tuples from $\A$ that generate finite substructures $A, B$, respectively, of $\A$, such that $g(\ov{a}), g(\ov{b})$ generate rigid substructures $A_0, B_0$, respectively, of $\B$.  If $(A_0,B_0)$ is a Ramsey duo for $\B$, then $(A,B)$ is a Ramsey duo for $\A$.  
%$\age(\B)$ consists of rigid elements and $\A$ is locally finite.  If $\B$ has RP, then $\A$ has RP.
\end{FmlaFree}

\begin{proof} Fix structures $\A, \B$, maps $g, f$, finite tuples $\ov{a}, \ov{b}$, and structures $A, B, A_0, B_0$, as in the statement.
Assume $(A_0,B_0)$ is a Ramsey duo for $\B$.

Fix finitely generated structures $A, B \subseteq \A$ and fix generators, $\ov{a}$ and $\ov{b}$ for these structures, respectively, in some fixed enumeration.  Since $\A$ is assumed locally finite, we may assume that $\ov{a}, \ov{b}$ enumerate $A, B$, respectively.

Fix a coloring $c: {\A \choose A} \raw 2$. 
We define the tuples $\ov{a}_0:=g(\ov{a}), \ov{b}_0:=g(\ov{b}), \ov{a}_1:=f(\ov{a}_0), \ov{b}_1:=f(\ov{b}_0)$.  
Moreover, let 
%$A:=\la \ov{a} \ra_\A, B:=\la \ov{b} \ra_\A$, 
$A_0:=\la \ov{a}_0 \ra_{\B}$, $B_0 = \la \ov{b}_0 \ra_\B$.  Since $fg$ is an embedding and $\ov{a}$ and $\ov{b}$ were assumed to be structures, $\ov{a}_1$ and $\ov{b}_1$ must also enumerate $\LA$-structures, which we denote by $A_1$ and $B_1$, respectively.  Since $g$ is not an embedding, we must consider that the structures generated by $\ov{a}_0, \ov{b}_0$ might be strictly larger than the generators:
$A_0:=\la \ov{a}_0 \ra_{\B}$, $B_0 = \la \ov{b}_0 \ra_\B$.
%We summarize our notation in a diagram:
%\begin{center}
%\begin{tikzcd}
%\ov{a}  \arrow[r,mapsto,"g"] &\ov{a}_0 \arrow[r,mapsto,"f"] & \ov{a}_1\\
%\ov{b} \arrow[r,mapsto,"g"] &\ov{b}_0 \arrow[r,mapsto,"f"] & \ov{b}_1
%\end{tikzcd}
%\end{center}
If there exists an $\LB$-isomorphism $\rho: A_0 \raw A_0'$, for some substructure $A_0' \subseteq \B$, let $\ov{a}_0' := \rho(\ov{a}_0)$ and define an induced coloring $c_0 : {\B \choose A_0} \raw 2$ by
$c_0(A_0') := c(\la f(\ov{a}_0') \ra_\A)$.
Since $\rho$ is an $\LB$-isomorphism, we know that $\ov{a}_0' \sim_\B \ov{a}_0$, and since $f$ is qftp-respecting, we know that 
$f(\ov{a}_0') \sim_\A f(\ov{a}_0)$ which implies that
$\la f(\ov{a}_0') \ra_\A \cong \la f(\ov{a}_0) \ra_\A \cong \A$.
This, and the fact that the isomorphism $\rho$, if it exists, is unique, guarantee that this coloring $c_0$ is well-defined and total on the domain ${\B \choose A_0}$.

By the assumption of that $(A_0, B_0)$ is a Ramsey duo, there is a copy $B_0'$ of $B_0$ in $\B$ homogeneous for the coloring $c_0$ on copies of $A_0$.  Thus, there is $d<2$ such that $c_0(A_0')=d$ for all $A_0' \cong_{\LB} A_0$ in $B_0'$.
%I DEFINED d
%
Since $B_0 \cong_{\LB} B_0'$,
Remark \ref{unique} guarantees the existence of a unique isomorphism $\sigma : B_0 \raw B_0'$.  Define
$\ov{b}_0' := \sigma(\ov{b}_0)$ and $\ov{b}_1' :=f(\ov{b}_0')$
and observe that
$\ov{b}_0' \sim_\B \ov{b}_0$.
%
%b'' is b_1'
%b' is b_1
%
By the qftp-respecting property of semi-retractions, since $\ov{b}_0' \sim_\B \ov{b}_0$,
$\ov{b}_1' = f(\ov{b}_0')\sim_\A f(\ov{b}_0)=\ov{b}_1$.
By the composition property of semi-retractions,
$\ov{b}_1 \sim_\A \ov{b}$.
Thus, by transitivity of the relation $\sim_\A$,
$\ov{b}_1' \sim_\A \ov{b}$.

Since $\ov{b}_1$ enumerates an $\LA$-structure the $\sim_\A$ tuples $\ov{b}_1'$ enumerates an $\LA$-structure, which we may denote by $B_1'$.  Moreover, we have shown that $B_1' \cong B$.
We claim that $B_1'$ is the desired homogeneous copy of $B$ in $\A$ for the coloring $c$.  To verify, let $A_1'$ be a copy of $A$ in $B_1'$, and let $\xi: A \rightarrow A_1'$ be the unique function witnessing the isomorphism.  Define $\ov{a}_1': = \xi(\ov{a})$ and note that $\ov{a}_1' \sim_\A \ov{a}$.  %As with the case of $\ov{b}_1'$, $\ov{a}_1'$ enumerates an $\LA$-structure, so we conclude that $\ov{a}_1'$ enumerates $A_1'$.
%Since $\ov{b}_0 \sim_\B \ov{b}_0'$ by \eqref{2.75} and we have assumed that $\ov{a}, \ov{b}$ enumerate finite substructures of $\A$,
Define $\ov{a}_0':=f^{-1}(\ov{a}_1')$.
By Proposition \ref{prop_fcn_restricted}, $\ov{b}_0'$ has the restricted inverse images under $f$ property for $\ov{a}$ witnessed by $\ov{a}_0$.   So in particular,  $\ov{a}_0' \subseteq \ov{b}_0'$ and $\ov{a}_0' \sim_{\B} \ov{a}_0$.
By homogeneity of $B_0'$ for the coloring $c_0$, $c_0(\la \ov{a}_0' \ra_\B) = d$, and since $\ov{a}_0' \sim_\B \ov{a}_0$,
%by a close inspection of the definition of $c_0$ (using rigidity of structures in $\age(\B)$)
$c_0(\la \ov{a}_0' \ra_\B) := c(\la f(\ov{a}_0') \ra_\A)$
thus
$d=c(\la f(\ov{a}_0') \ra_\A)=c(\la \ov{a}_1' \ra_\A)=c(A_1')$, as desired.

%We illustrate the maps as follows:
%
%\begin{center}
%\begin{tikzcd}
%\ov{a} \subseteq \ov{b} \arrow[r,mapsto,"g"] &\ov{a}_0 \subseteq \ov{b}_0 \arrow[r,mapsto,"f"] \arrow[d,dash,dashed,"\eqref{11}"]& \ov{a}_1 \subseteq \ov{b}_1 \\
%{} &\ov{a}_0' \subseteq \ov{b}_0' \arrow[r,mapsto,"f"] & \ov{a}_1' \subseteq \ov{b}_1' \arrow[llu,dash,dashed,controls={+(-1,3) and +(-1,2)},"\eqref{10}"]
%\end{tikzcd}
%\end{center}

\end{proof}

\begin{rmk}
In the previous argument for Corollary \ref{cor_fmla_free}, we make essential use of the local finiteness of $\A$.  If $\A$ is not locally finite, then some finite tuple $\ov{b}$ from $\A$ generates an infinite substructure $B \subseteq \A$.  Then $g(B) \subseteq |\B|$ is an infinite set that may or may not be generated by a finite set of generators from $\B$, in which case it may not be contained in an element of $\age(\B)$.  This is a problem, because any copy of $A \subseteq B$ in $fg(B) \subseteq \A$ has a preimage that is somewhere in $g(B)$, and may be outside the image $g(\ov{b})$ of the original generators, $\ov{b}$.  In general, the Ramsey property for $\B$ does not constrain infinite sets such as $g(B)$ that are not contained in elements of $\age(\B)$.  For a more specific illustration of what could go wrong, see Example \ref{succ}.
\end{rmk}

\section*{Acknowledgements} The second author thanks John Mumma for a helpful discussion about Frank Ramsey's 1929 paper and Anand Pillay for a discussion about lifting automorphisms in the early stages of thinking about this paper.
The authors thank the anonymous referee for comments that led to improvement  of the presentation of the paper.
\bibliographystyle{plainnat}
\bibliography{SR_References}

\end{document}